\documentclass[
 a4paper,
 11pt,
 reqno
]{amsart}

\usepackage[english]{babel}              
\usepackage[utf8]{inputenc}
\usepackage[T1]{fontenc}
\usepackage[final]{microtype}

\usepackage{lmodern}

\usepackage{amsmath}        
\usepackage{mathtools}      
\usepackage{amssymb}  
\usepackage{amsthm}
\usepackage{thmtools}
\usepackage{stackengine}
\usepackage{bbm}
\usepackage{dsfont}
\usepackage{mathrsfs}
\usepackage{nicematrix}

\numberwithin{equation}{section}

\usepackage{fancyhdr}
\usepackage{geometry}
\usepackage[dvipsnames]{xcolor}
\usepackage{graphicx}
\usepackage{float}
\usepackage{tikz}
\usetikzlibrary{hobby, shapes.misc, decorations.pathmorphing, patterns}
\usepackage{pgfplots}       
\pgfplotsset{compat=1.10, ticks=none}
\usepgfplotslibrary{fillbetween}
\definecolor{darkgreen}{rgb}{0.42, 0.62, 0.47}
\definecolor{paperblue}{rgb}{0.2, 0.4, 0.55}
\definecolor{brick}{rgb}{0.65, 0.3, 0.25}
\definecolor{paperyellow}{rgb}{0.9, 0.7, 0.3}

\newcommand{\abs}[1]{\left\lvert {#1} \right\rvert}
\newcommand{\norm}[1]{\left\lVert {#1} \right\rVert}

\newcommand{\rb}[1]{\left( {#1} \right)}
\renewcommand{\sb}[1]{\left[ {#1} \right]}
\newcommand{\cb}[1]{\left\{ {#1} \right\}}

\newcommand{\R}{\mathbb R}
\newcommand{\N}{\mathbb N}
\newcommand{\Z}{\mathbb Z}
\newcommand{\T}{\mathbb T}
\newcommand{\D}{\mathscr D}

\renewcommand{\S}{\mathbb S}
\newcommand{\Ha}{\mathcal H}
\newcommand{\1}{\mathds 1}
\newcommand{\defeq}{\vcentcolon =}
\newcommand{\loc}{\textup{loc}}
\newcommand{\sym}{\mathrm{sym}}
\newcommand{\starto}{\mathrel{\ensurestackMath{\stackon[3pt]{\rightharpoonup}{\scriptstyle\star}}}}

\newcommand{\xupref}[2]{\hspace{-0.3ex}\stackrel{\eqref{#1}}{#2}}

\DeclareMathOperator{\de}{\, d \hspace{- 2pt}}

\DeclareMathOperator{\spann}{span}
\DeclareMathOperator{\conv}{conv}
\DeclareMathOperator{\dive}{div}
\DeclareMathOperator{\spt}{spt}
\DeclareMathOperator{\reb}{\partial^* \hspace{- 2pt}}
\DeclareMathOperator{\GCD}{GCD}

\DeclareMathOperator{\GL}{GL}
\DeclareMathOperator{\Int}{Int}

\usepackage{url}
\usepackage[colorlinks = true, linkcolor= blue, citecolor = red, urlcolor = blue]{hyperref}
\usepackage[style = alphabetic, giveninits = true, maxbibnames = 4, maxalphanames = 4]{biblatex}
\renewbibmacro{in:}{}
\usepackage[nameinlink, capitalise, noabbrev]{cleveref}
\usepackage{csquotes}
\theoremstyle{plain}
\newtheorem*{thm*}{Theorem}
\newtheorem{thm}{Theorem}[section]
\newtheorem{lem}[thm]{Lemma}
\newtheorem{prop}[thm]{Proposition}
\newtheorem{cor}[thm]{Corollary}
\theoremstyle{definition}
\newtheorem{defi}[thm]{Definition}
\newtheorem*{not*}{Notation}

\newtheorem{rem}[thm]{Remark}

\title{The periodic Wulff problem in two dimensions}

\author[M. Bonacini]{Marco Bonacini}
\author[A. Fiorini]{Alberto Fiorini}

\address[M. Bonacini and A. Fiorini]{Department of Mathematics, University of Trento, Via Sommarive 14, 38123 Povo (Trento), Italy}
\email{marco.bonacini@unitn.it}
\email{alberto.fiorini@unitn.it}

\begin{document}

\begin{abstract}

    We characterize the minimizers of the anisotropic perimeter in the two-dimensional flat torus under an area constraint. 
    We show that, as the prescribed area increases, global minimizers undergo two phase transitions, changing from a droplet (a rescaled Wulff shape) 
    to a flat band wrapping around the torus in a preferred direction determined by the anisotropy (lamellar phase), and finally to a bubble (the complement of a reflected and rescaled Wulff shape).
    The transition thresholds are given explicitly in terms of the area and the lattice width of the Wulff shape.
    For certain non-symmetric surface tensions, the lamellar phase is absent altogether. Finally, we characterize the surface tensions for which further minimizers appear: 
    this may happen only in the lamellar regime, and exactly when the Wulff shape has a corner whose normal cone contains, in its interior, a direction realizing the lattice width; 
    in this case there are infinitely many minimizers.
    
\end{abstract}

\maketitle

\section{Introduction}

A classical theorem asserts that the solutions of the \emph{periodic isoperimetric problem} on the two-dimensional flat torus $\T^2 \defeq \R^2 / \Z^2$,
\begin{equation} \label{eq:1.isoper}
    \min \cb{P_{\T^2}(E) : E \subseteq \T^2, \ \abs{E} = M}, \qquad \qquad M \in (0, 1),
\end{equation}
are, depending on the area constraint $M$, either a disk, a horizontal or vertical strip (i.e. the region enclosed by two parallel 1-tori), or the complement of a disk. 
Here $\abs{E}$ is the Lebesgue measure of $E$, and $P_{\T^2} (E)$ is the perimeter of $E$ on the torus $\T^2$. 
A proof of this fact is given by Howards, Hutchings, and Morgan in \cite{HHM99}; see also \cite{CS06}.

More precisely, if $E$ is a minimizer for \eqref{eq:1.isoper},
then, depending on its area $M = \abs{E}$, the set $E$ takes one of the following forms, up to translation and Lebesgue-negligible sets (see \cref{fig:1.isomin}):
\begin{equation*}
	E = 
	\begin{cases}
		\sqrt{\frac{M}{\pi}} B_1 \qquad \quad & \text{ if } \quad 0 < M \le \frac{1}{\pi}, \\[1ex]
		L_M^{e_1} \text{ or } L_M^{e_2} \qquad \quad  & \text{ if } \quad \frac{1}{\pi} \le M \le 1 - \frac{1}{\pi}, \\[1ex]
		\rb{\sqrt{\frac{1 - M}{\pi}} B_1}^c \qquad \quad & \text{ if } \quad 1 - \frac{1}{\pi} \le M < 1,
	\end{cases}
\end{equation*}
where $B_1$ is the unit disk centered at the origin, and the horizontal and vertical lamellae are 
$L_M^{e_2} \defeq \T^1 \times (\ell_1, \ell_2)$ and $L_M^{e_1} \defeq (\ell_1, \ell_2) \times \T^1$, 
for some $\ell_1, \ell_2 \in (0, 1)$ with $\ell_2 - \ell_1 = M$.

In three dimensions, the minimizers of the periodic isoperimetric problem in the cubic flat torus were conjectured 
to be balls, circular tubes, lamellae, and their complements, see \cite{HPRR04, Ros05, CS06}; 
partial results were obtained in \cite{Had72, AFM13, Mil26} for small volume fractions and near half volume. 
A proof of this conjecture, with the classification of all minimizers, has been announced in a preprint released by OpenAI \cite{OAI26}.

\begin{figure}[t]
	\centering
	\begin{tikzpicture} [scale = 1.7]
		\draw[fill = darkgreen!35] (- 1, - 1) -- (- 1, 1) -- (1, 1) -- (1, - 1) -- cycle;
		\draw[fill = darkgreen, scale = 0.4] (0, 0) circle (1);
		\draw[thick, ->](- 1, - 0.00001) -- (- 1, 0.00001);
		\draw[thick, ->](1, - 0.00001) -- (1, 0.00001);
		\draw[thick, ->](- 0.00001, - 1) -- (0.00001, - 1);
		\draw[thick, ->](- 0.00001, 1) -- (0.00001, 1);
		\draw(1.3, 0.8) node{$\T^2$};
		\draw(0, 0) node{\scriptsize{$\sqrt{\frac{ M}{\pi}} B_1$}};
	\end{tikzpicture}
	\hspace{1 cm}
	\begin{tikzpicture} [scale = 1.7]
		\fill[darkgreen!35] (- 1, - 1) -- (- 1, 1) -- (1, 1) -- (1, - 1) -- cycle;
		\fill[darkgreen] (- 1, - 0.4) rectangle (1, 0.4);
		\draw (- 1, 0.4) -- (1, 0.4);
		\draw (- 1, - 0.4) -- (1, - 0.4);
		\draw (- 1, - 1) -- (- 1, 1) -- (1, 1) -- (1, - 1) -- cycle;
		\draw[thick, ->](- 1, - 0.00001) -- (- 1, 0.00001);
		\draw[thick, ->](1, - 0.00001) -- (1, 0.00001);
		\draw[thick, ->](- 0.00001, - 1) -- (0.00001, - 1);
		\draw[thick, ->](- 0.00001, 1) -- (0.00001, 1);
		\draw(1.3, 0.8) node{$\T^2$};
		\draw(0.5, - 0.1) node{$L_M^{e_2}$};
	\end{tikzpicture}
	\hspace{1 cm}
	\begin{tikzpicture} [scale = 1.7]
		\draw[fill = darkgreen] (- 1, - 1) -- (- 1, 1) -- (1, 1) -- (1, - 1) -- cycle;
		\draw[fill = darkgreen!35, scale = - 0.4] (0, 0) circle (1);
		\draw[thick, ->](- 1, - 0.00001) -- (- 1, 0.00001);
		\draw[thick, ->](1, - 0.00001) -- (1, 0.00001);
		\draw[thick, ->](- 0.00001, - 1) -- (0.00001, - 1);
		\draw[thick, ->](- 0.00001, 1) -- (0.00001, 1);
		\draw(1.3, 0.8) node{$\T^2$};
		\draw(0.1, 0.7) node{\scriptsize{$\rb{\sqrt{\frac{1 - M}{\pi}} B_1}^c$}};
	\end{tikzpicture}
	\caption{The minimizers of the isotropic periodic isoperimetric problem \eqref{eq:1.isoper} for increasing values of the area $M$.}
	\label{fig:1.isomin}
\end{figure}
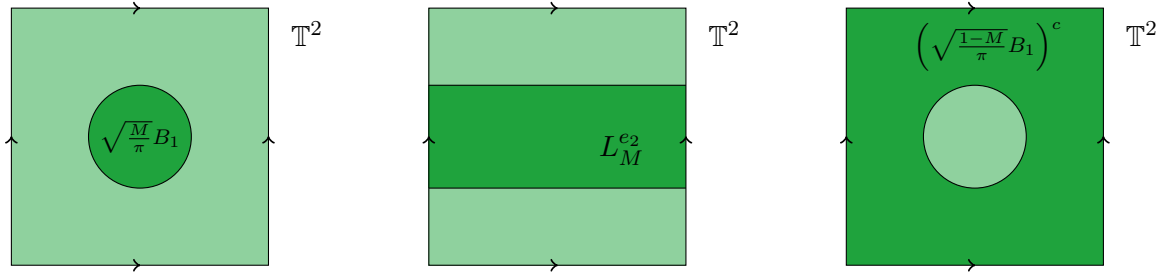

\medskip
The periodic isoperimetric problem describes the equilibrium shape of one phase occupying a prescribed fraction of a periodic cell. 
When the interface separates crystalline phases, its energy depends on its orientation, and the relevant problem becomes anisotropic. 
The goal of this paper is to obtain a complete description of the phase diagram of the anisotropic version of the periodic isoperimetric problem (\emph{periodic Wulff problem}), 
in two dimensions and for an arbitrary surface tension $\phi : \R^2 \to [0, + \infty)$:
\begin{equation} \label{eq:1.wulffper}
    \min \cb{P_{\T^2}^\phi(E) : E \subseteq \T^2, \ \abs{E} = M}, \qquad \qquad M \in (0, 1),
\end{equation}
where $P_{\T^2}^\phi(E)$ is the anisotropic perimeter of $E$ in the torus $\T^2$  
corresponding to the surface tension $\phi$ (see \cref{sec:2} for the relevant definitions).
Existence of minimizers for \eqref{eq:1.wulffper} follows from the direct method in the calculus of variations.

In the anisotropic setting, the picture described above remains generically valid, with the Wulff shape playing the role of the Euclidean ball. 
In particular, one expects a ``droplet phase'' for small values of $M$, where the minimizer is a rescaled copy of the Wulff shape; 
a ``lamellar phase'' for $M$ in a symmetric interval around $\frac{1}{2}$, where the minimizer is a band wrapping around the torus; 
and a ``bubble phase'' for values of $M$ close to $1$, where the minimizer is the complement of a rescaled and reflected Wulff shape.
However, the introduction of anisotropy in the surface tension produces new phenomena that in some cases make the picture qualitatively different from the isotropic setting. 
In particular:
\begin{enumerate}
    \item 
        In the isotropic case, the only two lamellar configurations minimizing the perimeter are the horizontal and vertical ones. 
        However, for a general anisotropy the preferred directions might be different, and one should also consider slanted lamellae (see \cref{defi:3.lamconf} and \cref{rem:4.lamdir});

    \item 
        For certain surface tensions, it might happen that lamellar configurations are \emph{never} global minimizers, for any value of the area parameter $M \in (0, 1)$: 
        the global minimizer transitions directly from the droplet phase to the bubble phase, skipping the lamellar phase (see \cref{rem:4.lamskip});

    \item 
        For certain non-regular surface tensions, in particular when the Wulff shape has corners,
        it is not true that the previous shapes are \emph{the only} possible minimizers: 
        in fact, there might exist infinitely many different solutions coexisting with the lamellae (see \cref{rem:4.laminf} and \cref{thm:uniqueness}).
        
\end{enumerate}

In order to state the main results of the paper, we first introduce some notation. 
For any integer direction $z \in \Z^2 \setminus \cb{0}$ one can define a slanted lamellar configuration $L_M^z$ of area $M$ in $\T^2$ 
whose orthogonal direction is precisely $\pm z$, see \cref{defi:3.lamconf}.
Given a convex body $K \subset \R^2$, we define its \emph{lattice width} by
\begin{equation} \label{eq:1.lattice-width}
    w(K) \defeq \min_{z \in \Z^2 \setminus \cb{0}} \rb{\max_{x, y \in K} z \cdot (x - y)},
\end{equation}
and we let 
\begin{equation}
    \D(K) \defeq \cb{z \in \Z^2 \setminus \cb{0} : w(K) = \max_{x, y \in K} z \cdot (x - y)} \ne \emptyset
\end{equation}
be the set of minimizing directions for $w(K)$. The quantity $w(K)$ measures the minimum, among integer directions $z\in\Z^2\setminus\cb{0}$, 
of the Euclidean distance between the two supporting lines of $K$ with normal $z$, weighted by $\abs{z}$.
Finally, we also define the quantity
\begin{equation} \label{eq:1.alphaK}
    \alpha(K) \defeq \frac{w(K)^2}{4 \abs{K}},
\end{equation}
which corresponds to the critical value of the area for which the anisotropic perimeter of a scaled Wulff shape with area $\alpha(K)$ 
equals the anisotropic perimeter of the lamella $L_{\alpha(K)}^z$, for $z\in\D(K)$.
The following is the main result of the paper, which classifies the solutions to \eqref{eq:1.wulffper} (see \cref{fig:1.wulffmin}).

\begin{thm} \label{thm:main}

    Let $\phi : \R^2 \to [0, + \infty)$ be a surface tension and let $K \subset \R^2$ be the corresponding Wulff shape.
    Then, depending on the area constraint $M \in (0, 1)$, the following sets solve the periodic Wulff problem \eqref{eq:1.wulffper}:
    \begin{equation*}
        \begin{dcases}
            \sqrt{\tfrac{M}{\abs{K}}} K \qquad & \text{ if } \quad 0 < M \le \alpha(K) \wedge \frac{1}{2}, \\[1ex]
            L_M^z \quad \text{ for } z \in \D(K) \qquad & \text{ if } \quad \alpha(K) \le M \le 1 - \alpha(K), \\[1ex]
            \rb{- \sqrt{\tfrac{1 - M}{\abs{K}}} K}^c \qquad & \text{ if } \quad 1 - \Big( \alpha(K) \wedge \frac{1}{2} \Big) \le M < 1.
        \end{dcases}
    \end{equation*}
    If $\phi$ is symmetric, then $\alpha(K) \le \frac{1}{2}$ and the lamella is always a minimizer for some value of $M$.
    
\end{thm}

\begin{figure}[t]
	\centering
	\begin{tikzpicture} [scale = 1.7]
		\draw[fill = darkgreen!35] (- 1, - 1) -- (- 1, 1) -- (1, 1) -- (1, - 1) -- cycle;
		\draw[fill = darkgreen, rotate = 45] (0, 0) ellipse (0.7 and 0.3);
		\draw[thick, ->](- 1, - 0.00001) -- (- 1, 0.00001);
		\draw[thick, ->](1, - 0.00001) -- (1, 0.00001);
		\draw[thick, ->](- 0.00001, - 1) -- (0.00001, - 1);
		\draw[thick, ->](- 0.00001, 1) -- (0.00001, 1);
		\draw(1.3, 0.8) node{$\T^2$};
		\draw(0, 0) node{\scriptsize{$\sqrt{\frac{M}{\abs{K}}} K$}};
	\end{tikzpicture}
	\hspace{1 cm}
	\begin{tikzpicture} [scale = 1.7]
		\fill[darkgreen!35] (- 1, - 1) -- (- 1, 1) -- (1, 1) -- (1, - 1) -- cycle;
		\fill[darkgreen] (- 1, - 1) -- (- 1, - 0.5) -- (0.5, 1) -- (1, 1) -- (1, 0.5) -- (- 0.5, - 1) -- cycle;
		\fill[darkgreen] (- 1, 1) -- (- 1, 0.5) -- (- 0.5, 1) -- cycle;
		\fill[darkgreen] (1, - 1) -- (1, - 0.5) -- (0.5, - 1) -- cycle;
		\draw (- 1, - 0.5) -- (0.5, 1);
		\draw (- 0.5, - 1) -- (1, 0.5);
		\draw (0.5, - 1) -- (1, - 0.5);
		\draw (- 0.5, 1) -- (- 1, 0.5);
		\draw (- 1, - 1) -- (- 1, 1) -- (1, 1) -- (1, - 1) -- cycle;
		\draw[thick, ->](- 1, - 0.00001) -- (- 1, 0.00001);
		\draw[thick, ->](1, - 0.00001) -- (1, 0.00001);
		\draw[thick, ->](- 0.00001, - 1) -- (0.00001, - 1);
		\draw[thick, ->](- 0.00001, 1) -- (0.00001, 1);
		\draw(1.3, 0.8) node{$\T^2$};
		\draw(0, 0) node{$L_M$};
	\end{tikzpicture}
	\hspace{1 cm}
	\begin{tikzpicture} [scale = 1.7]
		\draw[fill = darkgreen] (- 1, - 1) -- (- 1, 1) -- (1, 1) -- (1, - 1) -- cycle;
		\draw[fill = darkgreen!35, rotate = 45] (0, 0) ellipse (0.7 and 0.3);
		\draw[thick, ->](- 1, - 0.00001) -- (- 1, 0.00001);
		\draw[thick, ->](1, - 0.00001) -- (1, 0.00001);
		\draw[thick, ->](- 0.00001, - 1) -- (0.00001, - 1);
		\draw[thick, ->](- 0.00001, 1) -- (0.00001, 1);
		\draw(1.3, 0.8) node{$\T^2$};
		\draw(0.1, 0.7) node{\scriptsize{$\rb{\sqrt{\frac{1 - M}{\abs{K}}} K}^c$}};
	\end{tikzpicture}
	\caption{Solutions of the periodic Wulff problem \eqref{eq:1.wulffper} (in the case of  a symmetric surface tension): 
    minimizers transition from a droplet (rescaling of the Wulff shape) for small $M$, to a lamella for $M \approx 1/2$, 
    to a bubble (the complement of a rescaled Wulff shape) for $M$ close to $1$.}
	\label{fig:1.wulffmin}
\end{figure}
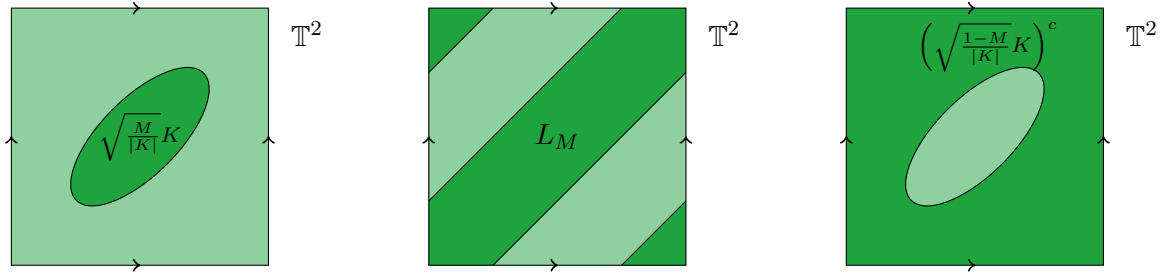

By the explicit computation of the anisotropic perimeter of the minimizers (see \cref{sec:3}), it follows that the \emph{anisotropic isoperimetric profile} in the torus $\T^2$ is
\begin{equation} \label{eq:1.wulffvalues}
	I^\phi_{\T^2}(M) \defeq \min_{\abs{E} = M} P_{\T^2}^\phi(E) =
	\begin{dcases}
		2 \sqrt{M \abs{K}} & \text{ if } \quad 0 < M \le \alpha(K) \wedge \frac{1}{2}, \\[1ex]
		w(K) & \text{ if } \quad \alpha(K) \le M \le 1 - \alpha(K), \\[1ex]
		2 \sqrt{(1 - M)\abs{K}} & \text{ if } \quad 1 - \Big( \alpha(K) \wedge \frac{1}{2} \Big) \le M < 1.
	\end{dcases}
\end{equation}
Moreover, we deduce the \emph{anisotropic isoperimetric inequality in $\T^2$}
\begin{equation} \label{eq:1.isop}
    C(K) \min \cb{\abs{E}, |\T^2 \setminus E|}^{1/2} \le P_{\T^2}^\phi(E)
\end{equation}
for any set of finite perimeter $E \subset \T^2$, where $C(K) = \min \{ 2 \abs{K}^{1/2}, \sqrt{2} w(K) \}$ is the optimal constant depending only on the Wulff shape $K$.

The condition $\alpha(K) \le M \le 1 - \alpha(K)$ in the statement should be interpreted as empty whenever $\alpha(K) > \frac{1}{2}$, in which case the lamellar phase is not present. 
This can happen only when $K$ is not symmetric: 
indeed, it is well known that $\alpha(K) \le \frac12$ for every symmetric convex body (\cref{prop:2.symarea}),
while for a general convex body one only has the weaker bound $\alpha(K) \le \frac{2}{3}$ (\cref{prop:2.asymarea}).
The sets $\sqrt{\tfrac{M}{\abs{K}}} K$, defined in $\R^2$, are canonically identified with their projections onto $\T^2$, 
which are well-defined since for values $M \le \alpha(K)$ the rescaled Wulff set $\sqrt{\frac{M}{\abs{K}}} K$ 
embeds injectively into the torus, as a consequence of the inequality in \cref{prop:2.torusemb}.

\cref{thm:main} is first proved for uniformly elliptic surface tensions, and is essentially based on the observation 
that the (Euclidean) mean curvature of the boundary of a minimizer for \eqref{eq:1.wulffper} has constant sign. 
This, combined with the Wulff inequality \eqref{eq:2.per-wulff} in $\R^2$, allows us to identify the possible candidate minimizers. 
The case of a general surface tension is then deduced by approximation.
In the special case where the anisotropy is uniformly elliptic, \cref{thm:main} was proved by the second author in \cite[Theorem~7.2]{Fio26}, 
under the more restrictive assumption that $\phi(e_2) + \phi(- e_2) < \phi(\nu) + \phi(- \nu)$ for any $\nu \in \S^1 \setminus \cb{\pm e_2}$ 
(in particular $\D(K) = \cb{\pm e_2}$, so that preferred lamellae are horizontal).

The proof of \cref{thm:main} shows that, for uniformly elliptic surface tensions, the previous classification of minimizers is exhaustive up to translations and Lebesgue-negligible sets.
However, without this regularity assumption, other minimizers might coexist with lamellar configurations.
In the following theorem we characterize the class of surface tensions for which this non-uniqueness phenomenon occurs.

\begin{thm} \label{thm:uniqueness}

Let $\phi : \R^2 \to [0, + \infty)$ be a surface tension and let $K \subset \R^2$ be the corresponding Wulff shape.
Then, problem \eqref{eq:1.wulffper} admits minimizers which are not equal, up to translations and Lebesgue-negligible sets, to any of the minimizers listed in \cref{thm:main} 
if and only if both of the following conditions hold:
    \begin{enumerate}
        \item 
            $\alpha(K) \le 1/2$ and $M \in [\alpha(K), 1 - \alpha(K)]$;
            
        \item 
            there exists $z_0 \in \D(K)$ such that $\phi$ is linear in a neighborhood of $z_0$.
            
    \end{enumerate}
    In this case there are infinitely many minimizers for the same value of $M$.
    
\end{thm}

The first condition shows that additional minimizers can occur only in the lamellar regime, including the transition values $M = \alpha(K)$ and $M = 1 - \alpha(K)$. 
The second condition can be interpreted as a geometric requirement on the Wulff shape: it is equivalent to the existence of a point $P \in \partial K$ admitting a nontrivial normal cone 
containing a minimizing lattice direction $z_0 \in \D(K)$ in its interior (see \cref{rem:phi_linear} for details).
In particular, if $\partial K$ is of class $C^1$, the minimizers in \cref{thm:main} are the only ones, up to translations and null sets.

The proof of \cref{thm:uniqueness} is given in \cref{sec:uniq}. 
While additional minimizers are constructed by a standard ``zig-zag'' argument (see \cref{rem:4.laminf}), 
the proof of the necessity of the two conditions in the statement combines a topological and a convexity argument. 
First, by exploiting the strict subadditivity of the isoperimetric profile \eqref{eq:1.wulffvalues}, we show that any minimizer $E$ and its complement $E^c$ are \textit{indecomposable} 
(the appropriate notion of connectedness in the framework of sets of finite perimeter; see \cite{ACMM01}). 
We then approximate $E$ by smooth sets $E_h$ such that both $E_h$ and its complement are connected; 
an Euler characteristic argument shows that either one of them is a topological disk, or both are annuli. 
In the first case, the rigidity of the Wulff inequality in $\R^2$ implies that either $E$ or $E^c$ coincides with a rescaled (possibly reflected) Wulff shape.
In the second case, the two boundary components of $E_h$ converge to two closed curves which wind around the torus in opposite directions, 
and their winding identifies a nonzero lattice vector $z_0$. 
Jensen's inequality then gives
\begin{equation*}
    P_{\T^2}^\phi(E) \ \ge \ \phi(z_0) + \phi(- z_0) \ \ge \ w(K),
\end{equation*}
while the reverse inequality follows by comparison with a lamella. 
This forces the outer normal to $E$ to lie, on each curve, in a cone where $\phi$ is linear. 
Therefore, either the normal is constant and $E$ is a lamella, or $\phi$ is linear in a neighborhood of $z_0$.

In \cref{sec:4} we also extend \cref{thm:main} and \cref{thm:uniqueness} to the flat torus $\R^2 / \Lambda$ associated with a general lattice $\Lambda$, 
by reducing it to the standard lattice through a linear change of variables (\cref{thm:4.main_latt} and \cref{thm:4.uniqueness_latt}).
Finally, we also discuss in \cref{sec:6} an application of the previous results to the global minimality of horizontal and vertical lamellae 
for an anisotropic version of the Ohta--Kawasaki energy \cite{OK86} with small nonlocal interaction, which has been recently considered in \cite{Fio26}.

\subsection*{Structure of the paper}

In \cref{sec:2} we provide rigorous definitions for the anisotropic perimeter associated with a surface tension, and we gather some preliminary facts from the geometry of numbers.
In \cref{sec:3}, after introducing the notation for slanted lamellae, we prove \cref{thm:main} first under the assumption of a uniformly elliptic surface tension, and then in the general case.
Additional remarks and examples are discussed in \cref{sec:4}.
The proof of the non-uniqueness result in \cref{thm:uniqueness} is given in \cref{sec:uniq}.
Finally, \cref{sec:6} contains an application to the anisotropic Ohta--Kawasaki functional.

\section{Preliminaries} \label{sec:2}

We gather here some tools which are needed in the sequel, from different areas of convex geometry.
In the following we consider directly the two-dimensional case, with the understanding that many definitions and results may be extended to any dimension.

\subsection{Convex bodies and surface tensions}

A set $K \subset \R^2$ is said to be a \textit{convex body} if it is compact, convex, and contains the origin in its interior, and it is said to be \textit{symmetric} if $K = - K$.
The symmetrized body of $K$ will be denoted by $K_\sym = \frac{1}{2}(K - K)$.
The \textit{support function} $\phi_K : \R^2 \to [0, + \infty)$ of $K$ is defined as
\begin{equation}\label{eq:2.support_fun}
    \phi_K(y) \defeq \max \Big\{ x \cdot y : x \in K \Big\}, \qquad \qquad y \in \R^2. 
\end{equation}
The support function satisfies the following properties:
\begin{enumerate}
    \item 
        $\phi_K(\lambda x) = \lambda \phi_K(x)$ for any $x \in \R^2$ and $\lambda \ge 0$ 
        (\textit{positive one-homogeneity});
        
    \item 
        $\phi_K(x + y) \le \phi_K(x) + \phi_K(y)$ for any $x, y \in \R^2$ (\textit{subadditivity});

    \item
        $\phi_K(x) > 0$ for any $x \in \R^2 \setminus \cb{0}$ (\textit{coercivity}).
        
\end{enumerate}
Under the assumption of positive one-homogeneity, subadditivity is actually equivalent to convexity,
and in particular the support function is continuous. 
As such, the last property is equivalent to requiring that there exists $\alpha > 0$ 
with $\phi_K(x) \ge \alpha \abs{x}$ for any $x \in \R^2$.
Furthermore, if $K$ is symmetric, then $\phi_K$ is a norm.

Any function $\phi : \R^2 \to [0, + \infty)$ satisfying (i), (ii) and (iii) is called a \textit{surface tension}. 
To any surface tension $\phi$ we associate the corresponding \textit{Wulff shape} $W_\phi$:
\begin{equation*}
    W_\phi \defeq \bigcap_{y \in \S^1} \cb{x \in \R^2 : x \cdot y \le \phi(y)},
\end{equation*}
which is a convex body whose support function is exactly $\phi$, giving a one-to-one correspondence between convex bodies and surface tensions.

\begin{rem} \label{rem:phi_linear}

    Given a convex body $K \subset \R^2$ with support function $\phi$ and $P \in \partial K$, we define the \emph{normal cone to $K$ at $P$} as the set
    \begin{equation*}
        N_K(P) \defeq \cb{\nu \in \R^2 : \nu \cdot (y - P) \le 0 \text{ for all } y \in K} = \cb{\nu \in \R^2 : \phi(\nu) = P \cdot \nu},
    \end{equation*}
    where the second identity follows by \eqref{eq:2.support_fun}. 
    It is then straightforward to check that, for $z_0 \in \R^2 \setminus \cb{0}$, $\phi$ is linear in a neighborhood of $z_0$ 
    if and only if $z_0$ lies in the interior of $N_K(P)$ for some $P\in\partial K$. 
    Notice that $N_K(P)$ has nonempty interior only if $P$ is a corner of $K$.
    
\end{rem}

For a surface tension $\phi$, we introduce its \textit{polar function} $\phi^\circ : \R^2 \to [0, + \infty)$:
\begin{equation*}
    \phi^\circ(x) \defeq \max \cb{x \cdot y : \phi(y) \le 1} = \max_{y \in \R^2 \setminus \cb{0}} \frac{x \cdot y}{\phi(y)}, \qquad \qquad x \in \R^2.
\end{equation*}
The polar function $\phi^\circ$ is likewise a surface tension and it satisfies $(\phi^\circ)^\circ = \phi$.
If $\phi$ is the support function of a convex body $K$, then $\phi^\circ$ is also called the \textit{distance function} of $K$, since it is seen that $K = W_\phi = \cb{ \phi^\circ \le 1}$.
Thus, the \textit{polar body} $K^\circ \subset \R^2$ of $K$ is defined as
\begin{equation*}
    K^\circ \defeq \cb{\phi \le 1} = \cb{x \in \R^2 : x \cdot y \le 1 \text{ for any } y \in K}.
\end{equation*}

In the planar case we have the following estimate of the product of the area of a symmetric convex body and its polar, due to Mahler \cite{Mah38}.

\begin{thm}[Mahler's inequality] \label{thm:2.mahler}

    For any symmetric convex body $K \subset \R^2$ it holds
    \begin{equation*}
        \abs{K} \abs{K^\circ} \ge 8
    \end{equation*}
    with equality if and only if $K$ is a parallelogram.

\end{thm}

In the asymmetric setting, for a general convex body $K \subset \R^2$, we have an analogue of Mahler's inequality involving the symmetrized body $K_\sym$, due to Eggleston \cite{Egg61}.

\begin{thm}[Eggleston's inequality] \label{thm:2.eggleston}

    For any convex body $K \subset \R^2$ it holds
    \begin{equation*}
        \abs{K} \abs{(K_\sym)^\circ} \ge 6, 
    \end{equation*}
    with equality if and only if $K$ is a triangle.
    
\end{thm}

\subsection{The anisotropic perimeter}

We recall some basic facts about the perimeter and its anisotropic counterpart, specifically in the case of the flat torus $\T^2$;
all the following properties are a consequence of well known results obtained in the classical Euclidean setting,
on noting that any set of finite perimeter $E \subset \T^2$ may be periodically extended 
to obtain a set of locally finite perimeter $E_{\R^2} \subset \R^2$
(see \cite{AFP00} and \cite{Mag12} for a complete treatise).

On $\T^2$ we have the usual function spaces $C^k(\T^2)$ for $k \in \N \cup \cb{+ \infty}$,
that may be identified with the subspaces of $C^k(\R^2)$
consisting of functions which are one-periodic along the coordinate directions 
(i.e. functions $u : \R^2 \to \R$ with $u(x + e_1) = u(x + e_2) = u(x)$).

A measurable set $E \subset \T^2$ has \textit{finite perimeter} if its characteristic function $\1_E$ has finite total variation in $\T^2$, namely
\begin{equation*}
    P_{\T^2}(E) \defeq \sup \cb{\int_E \dive T : T \in C^1(\T^2; \R^2), \ \abs{T} \le 1} < + \infty.
\end{equation*}
For a set of finite perimeter $E \subset \T^2$, the \textit{Generalized Divergence Theorem} holds:
\begin{equation*}
    \int_E \dive T = \int_{\reb E} T \cdot \nu_E \de \Ha^1 \qquad \qquad \text{for any } T \in C^1(\T^2; \R^2),
\end{equation*}
where $\reb E \subset \T^2$ is the reduced boundary of $E$, $\nu_E : \reb E \to \S^1$ is the outer unit normal to $E$ and $\Ha^1$ is the $1$-dimensional Hausdorff measure.

We may then define the anisotropic perimeter.

\begin{defi}[Anisotropic perimeter]

    Let $\phi : \R^2 \to [0, + \infty)$ be a surface tension.
    For any set of finite perimeter $E \subset \T^2$,
    the $\phi$\textit{-surface energy} (or \textit{anisotropic perimeter}) of $E$ in $\T^2$ is given by
    \begin{equation*}
        P_{\T^2}^\phi(E) \defeq \int_{\reb E} \phi(\nu_E(x)) \de \Ha^1 (x).    
    \end{equation*}
    
\end{defi}

Throughout the remainder of the paper, we denote the standard anisotropic perimeter in $\R^2$ by $P^\phi$, 
to distinguish it from the anisotropic perimeter in the torus $P_{\T^2}^\phi$, which incorporates the periodicity conditions.
We note that, denoting by $\pi : \R^2 \to \T^2$ the natural projection,
\begin{equation} \label{eq:2.relper}
    P_{\T^2}^\phi(E) = P^\phi(\pi^{- 1}(E); [0, 1)^2), \quad \text{where} \quad P^\phi(F;B) \defeq \int_{B \cap \reb F} \phi(\nu_F(x)) \de \Ha^1(x)
\end{equation}
for any set of finite perimeter $F \subset \R^2$ and any Borel set $B \subset \R^2$.

The anisotropic version of the classical isoperimetric problem in $\R^2$ is known as the \textit{Wulff problem}:
minimize the anisotropic perimeter $P^\phi$ among all sets $E \subset \R^2$ with area $\abs{E} = M$ for a fixed constant $M > 0$, i.e., 
\begin{equation} \label{eq:2.wulff}
    \inf \cb{P^\phi(E) : E \subset \R^2, \ \abs{E} = M}.
\end{equation}
The unique solution (up to translations) of this problem is given by the Wulff shape $W_\phi$ of $\phi$, suitably rescaled to satisfy the area constraint (see \cite[Chapter~20]{Mag12} and references therein).
This can be equivalently stated through the \textit{Wulff inequality}: 
for every set $E \subset \R^2$ of finite perimeter and area,
\begin{equation} \label{eq:2.per-wulff}
    P^\phi(E) \ge 2 \sqrt{\abs{W_\phi} \abs{E}},
\end{equation}
with equality if and only if, up to translations and dilations, $E$ is equal to the Wulff shape $W_\phi$.

\subsection{Geometry of numbers}

Our analysis will make use of classical tools from the geometry of numbers, particularly regarding the behavior of convex bodies with respect to the integer lattice $\Z^2$. 
A comprehensive treatise of this background material can be found in the opening chapters of Gruber and Lekkerkerker \cite{GL87}.

\begin{lem}[Bézout's lemma]

    Let $p, q \in \Z$ be integers with greatest common divisor (denoted by $\GCD(p, q)$) equal to $d$.
    Then there exist integers $m, n \in \Z$ such that $p m + q n = d$.
    
\end{lem}

In the following we will often consider a coprime pair $(p, q) \in \Z^2 \setminus \cb{0}$, meaning that $\GCD(p, q) = 1$.
In this case, Bézout's lemma ensures the existence of integers $m, n \in \Z$ with $p m + q n = 1$.

Let us also recall the definition of a lattice.

\begin{defi}

    A set $\Lambda \subset \R^2$ is said to be a \textit{lattice} if there exists a matrix $A_\Lambda \in \GL(\R^2)$ such that $\Lambda = A_\Lambda \Z^2$,
    and its \textit{determinant} is defined as $d(\Lambda) \defeq \abs{\det A_\Lambda}$.
    
\end{defi}

The lattice $\Z^2$ is also called \textit{standard lattice}, and its determinant is evidently equal to $1$.

\begin{rem}

    The determinant of a lattice $\Lambda$ is a well-defined quantity which does not depend on the choice of the matrix $A_\Lambda \in \GL(\R^2)$:
    indeed, let $A_1, A_2 \in \GL(\R^2)$ be such that $A_1 \Z^2 = A_2 \Z^2$, or equivalently, $\Z^2 = A_1^{- 1} A_2 \Z^2$;
    setting $U = A_1^{- 1} A_2 \in \GL(\R^2)$, we prove that $\det U = \pm 1$ to conclude.
    Since $U \Z^2 = \Z^2$ we easily see that $U$ has integer entries, and as a consequence its determinant must be a nonzero integer, 
    but then the very same argument holds for $U^{- 1}$, therefore the determinant must be equal to $\pm 1$.
    
\end{rem}

We say that a matrix $U \in \GL(\R^2)$ is \textit{unimodular} if its entries are integers and its determinant is equal to $\pm 1$,
and we denote by $\GL(\Z^2)$ the subgroup of $\GL(\R^2)$ made of unimodular matrices.
It is worth noting that the integer lattice $\Z^2$ is invariant under unimodular transformations: $U \Z^2 = \Z^2$ for every unimodular matrix $U \in \GL(\Z^2)$.

Given a lattice $\Lambda = A_\Lambda \Z^2$ for some $A_\Lambda \in \GL(\R^2)$, we define its \textit{polar lattice} $\Lambda^\circ \defeq A_\Lambda^{- T} \Z^2$, 
where $A_\Lambda^{- T}$ denotes the transpose of the inverse of $A_\Lambda$. 
As a simple consequence, the determinant of the polar lattice is given by $d(\Lambda^\circ) = 1/d(\Lambda)$.
Notice also that $(\Z^2)^\circ = \Z^2$.

Next, we need to introduce the notion of successive minima.

\begin{defi}[Successive minima]

    Let $K \subset \R^2$ be a convex body and $\Lambda \subset \R^2$ be a lattice. 
    The \textit{$i$-th successive minimum of $K$ with respect to $\Lambda$} is defined as
    \begin{equation*}
        \lambda_i(K, \Lambda) \defeq \min \Big\{ \lambda > 0 : \dim \big( \spann(\lambda K \cap \Lambda) \big) \ge i \Big\}, \qquad \qquad i = 1, 2.
    \end{equation*}
    
\end{defi}

The $i$-th successive minimum is the smallest dilation factor $\lambda > 0$ 
for which $\lambda K$ contains $i$ linearly independent lattice points of $\Lambda$ (see \cref{fig:2.latticefsm}).
In particular, we have that $\lambda_1 \le \lambda_2$.
We will mostly focus on the first successive minimum, which may be written as
\begin{equation*}
    \lambda_1(K, \Lambda) = \min \Big\{ \lambda > 0 : \lambda K \cap \Lambda \ne \cb{0} \Big\}.
\end{equation*}

\begin{figure}[t]
    \centering
    \begin{tikzpicture} [scale = 2]       
        \draw[fill = darkgreen!35, scale = 1.66] (- 0.5, - 0.75) -- (1, 0) -- (0.5, 0.75) -- (- 1, 0) -- cycle;
        \draw[fill = darkgreen!60] (- 0.5, - 0.75) -- (1, 0) -- (0.5, 0.75) -- (- 1, 0) -- cycle;
        \draw[fill = darkgreen, scale = 0.6] (- 0.5, - 0.75) -- (1, 0) -- (0.5, 0.75) -- (- 1, 0) -- cycle;
        \draw[thin, dashed] (- 2, - 1) grid (2, 1);
        \foreach \i in {- 2,..., 2} {
            \foreach \j in {- 1,..., 1} {
                \fill [blue!60!black] (\i, \j) circle (0.05) node [below]{\color{black}{\scriptsize{$(\i, \j)$}}};
            }
        }
        \draw[-latex, thick] (- 2.5, 0) -- (2.5, 0) node[below]{$x_1$};
        \draw[-latex, thick] (0, - 1.5) -- (0, 1.5) node[left]{$x_2$};
        \draw (0.25, 0.2) node {\scriptsize{$K$}};
        \draw (0.4, 0.55) node {\scriptsize{$\lambda_1 K$}};
        \draw (1, 1.3) node {\scriptsize{$\lambda_2 K$}};
    \end{tikzpicture}
    \caption{A convex body $K \subset \R^2$ and its dilations, with the successive minima $\lambda_1$ and $\lambda_2$ as dilation factors, for the standard lattice $\Z^2$.}
    \label{fig:2.latticefsm}
\end{figure}
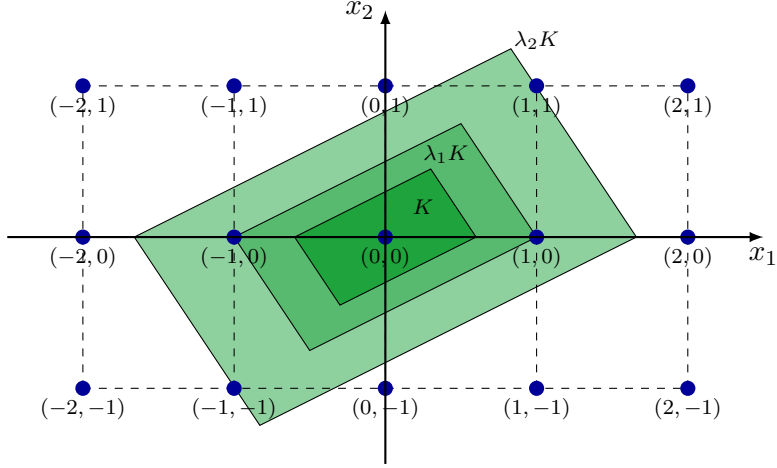

We now introduce the notion of lattice width, which provides another geometric interpretation of the first successive minimum.

\begin{defi}[Lattice width] \label{defi:2.lattice_width}

    Let $K \subset \R^2$ be a convex body and $\Lambda \subset \R^2$ be a lattice.
    The \textit{lattice width $w_\Lambda (K)$ of $K$ with respect to $\Lambda$} is defined as
    \begin{equation*}
        w_\Lambda (K) \defeq \min_{z \in \Lambda^\circ \setminus \cb{0}} \rb{\max_{x, y \in K} z \cdot (x - y)} 
        = \min_{z \in \Lambda^\circ \setminus \cb{0}} \big( \phi_K(z) + \phi_K(- z) \big),
    \end{equation*}
    where $\phi_K$ is the support function of $K$. We denote the set of all minimizing lattice directions for $K$ by
    \begin{equation*}
        \D_\Lambda(K) \defeq \Big\{ z \in \Lambda^\circ \setminus \cb{0} : w_\Lambda (K) = \phi_K(z) + \phi_K(- z) \Big\}.
    \end{equation*}
    
\end{defi} 

\begin{rem}

    The lattice width is related to the first successive minimum of $K$ by the identity
    \begin{equation} \label{eq:2.firstwidth}
        w_\Lambda (K) = 2 \lambda_1((K_\sym)^\circ, \Lambda^\circ).
    \end{equation}
    Indeed, denoting by $\phi_\sym$ the support function of $K_\sym$, we have that $\phi_\sym(x) = \frac{\phi(x) + \phi(- x)}{2}$ for $x \in \R^2$, so that
    \begin{align*}
        w_\Lambda(K) &= 2 \min_{z \in \Lambda^\circ \setminus \cb{0}} \phi_\sym(z) 
        = 2 \min \Big\{ \lambda > 0 : \text{ there exists } z \in \Lambda^\circ \setminus \cb{0} \text{ with } \phi_\sym(z) \le \lambda \Big\} \\
        &= 2 \min \Big\{ \lambda > 0 : \text{ there exists } z \in \lambda (K_\sym)^\circ \cap \Lambda^\circ \setminus \cb{0} \Big\} = 2 \lambda_1((K_\sym)^\circ, \Lambda^\circ).
    \end{align*}

\end{rem}

\begin{rem} \label{rem:2.latticetrans}

    The case of a general lattice $\Lambda = A_\Lambda \Z^2$ for some $A_\Lambda \in \GL(\R^2)$ can be traced back to the standard lattice $\Z^2$:
    indeed given a convex body $K \subset \R^2$ we have that
    \begin{align*}
        \lambda_1(K, \Lambda) &= \min \Big\{ \lambda > 0 : \lambda K \cap \Lambda \ne \cb{0} \Big\} \\
        &= \min \Big\{ \lambda > 0 : \lambda A_\Lambda^{- 1} K \cap \Z^2 \ne \cb{0} \Big\} = \lambda_1 (A_\Lambda^{- 1} K, \Z^2),
    \end{align*}
    and in turn
    \begin{align*}
         w_\Lambda(K) &= 2 \lambda_1((K_\sym)^\circ, \Lambda^\circ) = 2 \lambda_1(A_\Lambda^T (K_\sym)^\circ, \Z^2) \\
         &= 2 \lambda_1((A_\Lambda^{- 1} K_\sym)^\circ, \Z^2) = 2 \lambda_1(((A_\Lambda^{- 1} K)_\sym)^\circ, \Z^2) = w_{\Z^2}(A_\Lambda^{- 1} K).
    \end{align*}
    As a consequence, the first successive minimum of a convex body with respect to the integer lattice $\Z^2$, as well as its lattice width, is invariant under unimodular transformations.
    For later use we also remark that
    \begin{equation*}
        \abs{A_\Lambda^{- 1} K} = \frac{\abs{K}}{\abs{\det A_\Lambda}} = \frac{\abs{K}}{d(\Lambda)}.
    \end{equation*}
    
\end{rem}

Henceforth, if $\Lambda = \Z^2$ we simply write $\lambda_1(K)$, $w(K)$, and $\D(K)$ in place of $\lambda_1(K, \Z^2)$, $w_{\Z^2}(K)$, and $\D_{\Z^2}(K)$, respectively.

The following classical result is due to Minkowski, see \cite[Section~9]{GL87}.

\begin{thm}[Minkowski's second theorem]

    Let $K \subset \R^2$ be a symmetric convex body.
    Then
    \begin{equation*}
        \lambda_1(K) \lambda_2(K) \abs{K} \le 4. 
    \end{equation*}
    
\end{thm}

Next, we state and prove a sharp lower bound on the area of a convex body in terms of its lattice width, 
both in the symmetric (\cref{prop:2.symarea}) and the non-symmetric (\cref{prop:2.asymarea}) cases, 
together with a characterization of the convex bodies for which equality is attained (see \cref{fig:2.symeq}). 
These results are well known to experts in the geometry of numbers, but we include proofs for completeness.

\begin{prop} \label{prop:2.symarea}

    Let $K \subset \R^2$ be a symmetric convex body. Then
    \begin{equation} \label{eq:2.symarea}
        \abs{K} \ge \frac{w(K)^2}{2},
    \end{equation}
    and equality holds if and only if, up to dilations and unimodular transformations, $K$ is a parallelogram of the form
    \begin{equation*}
        P_s \defeq \conv \Big\{ \pm (s, 1), \pm (1, 0) \Big\}, \qquad s \in [0, 1).
    \end{equation*}

\end{prop}

\begin{figure}[ht]
    \centering
    \begin{tikzpicture} [scale = 1.5]
        \draw[fill = darkgreen] (0, - 1) -- (- 1, 0) -- (0, 1) -- (1, 0) -- cycle;
        \draw(0.75, 0.75) node{$P_0$};
        \foreach \x in {- 1, 1}
            \draw (\x, 0.1) -- (\x, - 0.1) node[below] {\small \x};
        \foreach \y in {- 1, 1}
            \draw (0.1, \y) -- (- 0.1, \y) node[left] {\small \y};
        \draw[-latex, thick] (- 1.5, 0) -- (1.5, 0) node[below]{$x_1$};
        \draw[-latex, thick] (0, - 1.5) -- (0, 1.5) node[left]{$x_2$};
    \end{tikzpicture}
    \hspace{0.2 cm}
    \begin{tikzpicture} [scale = 1.5]
        \draw[fill = darkgreen] (- 0.5, - 1) -- (- 1, 0) -- (0.5, 1) -- (1, 0) -- cycle;
        \draw(1, 1) node{$P_{1/2}$};
        \foreach \x in {- 1, 1}
            \draw (\x, 0.1) -- (\x, - 0.1) node[below] {\small \x};
        \foreach \y in {- 1, 1}
            \draw (0.1, \y) -- (- 0.1, \y) node[left] {\small \y};
        \draw[-latex, thick] (- 1.5, 0) -- (1.5, 0) node[below]{$x_1$};
        \draw[-latex, thick] (0, - 1.5) -- (0, 1.5) node[left]{$x_2$};
    \end{tikzpicture}
    \hspace{0.2 cm}
    \begin{tikzpicture} [scale = 1.5]
        \draw[fill = darkgreen] (0, 1) -- (- 1, - 1) -- (1, 0) -- cycle;
        \draw(0.75, 0.75) node{$T$};
        \foreach \x in {- 1, 1}
            \draw (\x, 0.1) -- (\x, - 0.1) node[below] {\small \x};
        \foreach \y in {- 1, 1}
            \draw (0.1, \y) -- (- 0.1, \y) node[left] {\small \y};
        \draw[-latex, thick] (- 1.5, 0) -- (1.5, 0) node[below]{$x_1$};
        \draw[-latex, thick] (0, - 1.5) -- (0, 1.5) node[left]{$x_2$};
    \end{tikzpicture}    
    \caption{Two parallelograms realizing the equality in \eqref{eq:2.symarea}, and the triangle $T$ achieving the equality in \eqref{eq:2.asymarea}.}
    \label{fig:2.symeq}
\end{figure}
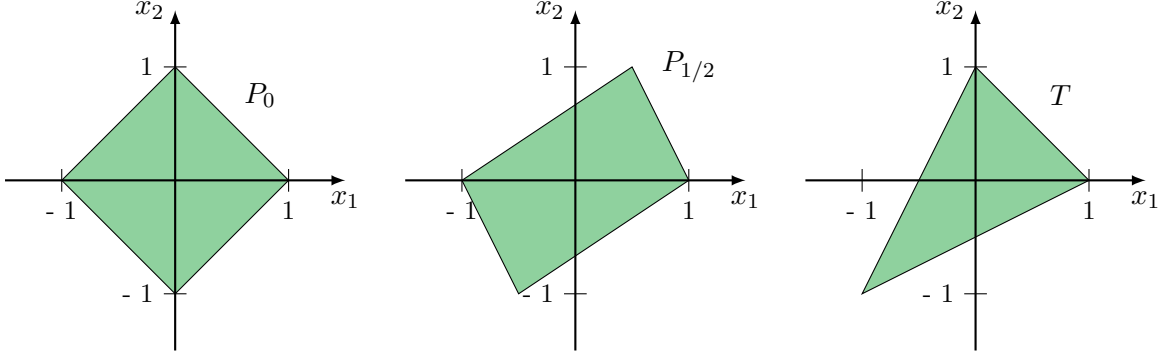

For this result we give two different proofs: the first one uses classical inequalities from the geometry of numbers,
whereas the second is an explicit construction which allows one to characterize the equality case in a precise manner.

\begin{proof}[First proof of \cref{prop:2.symarea}]

    By Minkowski's second theorem and Mahler's inequality we have
    \begin{equation*}
        8 \le \abs{K} \abs{K^\circ} \le \frac{4 \abs{K}}{\lambda_1(K^\circ) \lambda_2(K^\circ)},
    \end{equation*}
    which implies that
    \begin{equation*}
        \abs{K} \ge 2 \lambda_1(K^\circ) \lambda_2(K^\circ) \ge 2 \lambda_1(K^\circ)^2 \xupref{eq:2.firstwidth}{=} \frac{w(K)^2}{2}.
    \end{equation*}
    If the equality holds, by the characterization of the equality case in Mahler's inequality, $K$ must be a parallelogram. 
    In turn, one can explicitly characterize the parallelograms which realize the equality in \eqref{eq:2.symarea}, as discussed in the second proof below. 
    \qedhere

\end{proof}

\begin{proof}[Second proof of \cref{prop:2.symarea}]

    We divide the proof into two steps.
    In the first step we establish the inequality under the assumption that $e_2 \in \D(K)$, meaning that the vertical direction minimizes the lattice width. 
    The core strategy is to inscribe into $K$ a symmetric convex polygon $H$---which is generically a hexagon but might degenerate into a parallelogram---
    and bound its area from below in terms of the square of the lattice width. 
    We select two opposite vertices $\pm (a, b)$ on $\partial K$ where the tangent support lines are horizontal; 
    the remaining two pairs of opposite vertices are then chosen as extreme points of $K$ along integer directions $(1, - m)$, $(1, - m - 1)$ 
    which are ``almost'' orthogonal to the vector $(a, b)$.
    In the second step we obtain the general case thanks to Bézout's lemma.

    \vspace{5 pt}
    \noindent
    \textit{Step one}:
    let $K$ be a symmetric convex body with $e_2 \in \D(K)$, so that $w(K) = 2 \phi(e_2)$, where $\phi = \phi_K$ denotes the support function of $K$.
    Then, there exists a vector $(a, b) \in K$ such that $\phi(\pm e_2) = (\pm a, \pm b) \cdot (\pm e_2) = b > 0$.
    We may also choose $a \in \R$ such that $a \ge a'$ whenever $(a', b) \in K$.
    Then, there exists $m \in \Z$ with
    \begin{equation} \label{eq:2.symineqab}
        m b \le a < (m + 1) b.
    \end{equation}
    By assumption, $\phi(e_2) \le \phi(1, - m)$ and $\phi(e_2) \le \phi(1, - m - 1)$, therefore
    \begin{align}
        b &= \phi(e_2) \le \phi(1, - m) = \max_{(x_1, x_2) \in K} (x_1 - m x_2) = c - m d, \label{eq:2.symineq1} \\
        b &= \phi(e_2) \le \phi(1, - m - 1) = \max_{(x_1, x_2) \in K} (x_1 - (m + 1) x_2) = e - (m + 1) f, \label{eq:2.symineq2}
    \end{align}
    for some $(c, d), (e, f) \in K$. Notice that, by definition of $b$, we immediately deduce $\abs{d}, \abs{f} \le b$. 
    Moreover, we also have $d, f < b$: indeed, if it were $d = b$, then by \eqref{eq:2.symineq1} we would have $c \ge (m + 1) b > a$ 
    and the pair $(c, b) \in K$ would violate the initial choice of $a$; similarly if $f = b$.
    
    By \eqref{eq:2.symineqab} the pair $(c, d)$ must be different from $\pm (a, b)$, and the pair $(e,f)$ must be different from $(a,b)$ (but could possibly coincide with $(-a,-b)$). 
    Therefore we have at least two pairs of opposite points and the convex hull $H = \conv \cb{\pm (a, b), \pm (c, d), \pm (e, f)}$ is a nondegenerate symmetric polygon 
    (either a hexagon or a parallelogram in case $(e,f) \in \conv \cb{\pm (a, b), \pm (c, d)}$).
    In order to prove the inequality \eqref{eq:2.symarea}, we show that $\abs{H} \ge 2 b^2$.
    We have three different cases to study.
    \begin{itemize}
        \item 
            If $d \le 0$, by \eqref{eq:2.symineqab} and \eqref{eq:2.symineq1} we observe that
            \begin{equation*}
                b c - a d \ge m d b + b^2 - a d = b^2 - d (a - m b) \ge b^2,
            \end{equation*}
            which in turn implies that the parallelogram $P = \conv \cb{\pm (a, b), \pm (c, d)}$ satisfies
            \begin{equation*}
                \abs{H} \ge \abs{P} = 2 \abs{a d - b c} = 2 (b c - a d) \ge 2 b^2.
            \end{equation*}
            Moreover, equality holds if and only if $H = P$, and either $d = 0$ and $c = b$, or $a = m b$ and $c = b + m d$.

        \item 
            If $f > 0$, by means of \eqref{eq:2.symineqab} and \eqref{eq:2.symineq2} we similarly have that
            \begin{equation*}
                b e - a f \ge (m + 1) f b + b^2 - a f = b^2 + f ((m + 1) b - a) > b^2, 
            \end{equation*}
            so that the parallelogram $P = \conv \cb{\pm (a, b), \pm (e, f)} $ satisfies
            \begin{equation*}
                \abs{H} \ge \abs{P} = 2 \abs{a f - b e} = 2 (b e - a f) > 2 b^2
            \end{equation*}
            (notice that $P$ is non degenerate as the case $(e, f) = (- a, - b)$ is excluded by the assumption $f > 0$). Moreover, equality can never be attained.

        \item 
            If instead $d > 0$ and $f \le 0$, we see that
            \begin{align*}
                \abs{H} &= (- a f + b e) + (e d - f c) + (b c - a d) = - a f - a d + (b + d) e + (b - f) c \\
                &\ge - a f - a d + (b + d) ((m + 1) f + b) + (b - f) (m d + b) \\
                &= 2 b^2 - a f - a d + (m + 1) b d + m f b + d f,
            \end{align*}
            recalling that $b \ge \abs{d}, \abs{f}$.
            Notice that if $d + f \le 0$ one has
            \begin{equation*}
                - a f - a d + (m + 1) b d + m f b + d f = (d + f) (m b - a) + d (b + f) \ge 0,
            \end{equation*}
            whereas, if $d + f > 0$, then
            \begin{equation*}
                - a f - a d + (m + 1) b d + m f b + d f = (d + f) ((m + 1) b - a) - f (b - d) > 0.
            \end{equation*}
            We conclude again that $\abs{H} \ge 2 b^2$.
            
            To discuss the equality case, if $d + f \le 0$ then $f = - b$ and either $d = - f = b$ (which is impossible, since $d<b$) or $a = m b$. 
            Moreover, $e = (m + 1) f + b = - m b = - a$, that is, we are in the case $(e, f) = (- a, - b)$. 
            Hence $H = \conv \cb{\pm (a, b),\pm (c, d)}$ is a parallelogram with $a = mb$ and $c = b + m d$.
            
            Instead, if $d + f > 0$, it is evident that the equality cannot be attained.
            
    \end{itemize}
    As a consequence, since $H \subset K$ by convexity of the latter, we have 
    \begin{equation*}
        \abs{K} \ge \abs{H} \ge 2 b^2 = \frac{w(K)^2}{2}.
    \end{equation*}
    By the previous analysis, equality holds if and only if $K=H=\conv\cb{\pm (a, b),\pm(c, d)}$ is a parallelogram such that
    either $d = 0$ and $c = b$, or $a = m b$ and $c = b + m d$. In particular, up to a rescaling $K$ belongs to one of the following two families of parallelograms:
    \begin{equation*}
        P_s \defeq \conv \cb{\pm (s, 1), \pm (1, 0)}, \qquad \qquad P_t^m \defeq \conv \cb{\pm (m, 1), \pm (1 + m t, t)},
    \end{equation*}
    for $t, s \in \R$, $m \in \Z$. Notice that any parallelogram $P_t^m$ of the second family can be transformed into a parallelogram of the first family by the unimodular matrix
    $U = \big( \begin{smallmatrix} - 1 & m + 1 \\ 1 & - m \end{smallmatrix} \big)$.
    Moreover, any parallelogram $P_s$ with $\abs{s} \ge 1$ of the first family can be transformed into a parallelogram $P_{s'}$ with $\abs{s'} < 1$ by the unimodular matrix
    $U = \big( \begin{smallmatrix} 1 & - \lfloor s \rfloor \\ 0 & 1 \end{smallmatrix} \big)$, where $\lfloor s \rfloor$ is the integer part of $s$.
    Finally, if $s' < 0$ a reflection about the vertical axis (which is again a unimodular transformation) maps $P_{s'}$ into $P_{- s'}$.

    \vspace{5 pt}
    \noindent
    \textit{Step two}:
    suppose that $(p_0, q_0) \in \D(K)$.
    Since $p_0$ and $q_0$ are coprime, by Bézout's lemma there exist $m, n \in \Z$ with $p_0 m + q_0 n = 1$. 
    Consider the unimodular matrix $U \in \GL(\Z^2)$ defined as
    \begin{equation} \label{eq:2.proof_sym_1}
        U \defeq
        \begin{pmatrix}
            n & -m \\
            p_0 & q_0
        \end{pmatrix}.
    \end{equation}
    Then $U K$ is a symmetric convex body with $\phi_{U K}(y) = \phi_K(U^T y)$
    and, since $U^T e_2 = \big( \begin{smallmatrix} p_0 \\ q_0 \end{smallmatrix} \big)$, it follows that $e_2 \in \D(U K)$. 
    Hence by Step~one  
    \begin{equation*}
        \abs{K} = \abs{U K} \ge \frac{w(U K)^2}{2} = \frac{w(K)^2}{2},
    \end{equation*}
    where we used the fact that the lattice width is invariant under unimodular transformations, as discussed in \cref{rem:2.latticetrans}. 
    Eventually, equality holds whenever $K$ coincides with the unimodular image of one of the parallelograms characterized in Step one.
    \qedhere
    
\end{proof}

If $K$ is not symmetric, the inequality \eqref{eq:2.symarea} might fail: consider for instance the triangle $T = \conv \cb{(- 1, - 1), (1, 0), (0, 1)}$.
In the general case we have the following theorem, originally proved in \cite{FM74,Mak78} (see also \cite[Theorem~1.1]{HX19} for a further refinement).

\begin{prop} \label{prop:2.asymarea}

    Let $K \subset \R^2$ be a convex body. Then
    \begin{equation} \label{eq:2.asymarea}
        \abs{K} \ge \frac{3}{8} \, w(K)^2,
    \end{equation}
    and equality holds if and only if, up to dilations, translations, and unimodular transformations, $K$ is equal to the triangle
    \begin{equation*}
        T \defeq \conv \Big\{ (- 1, - 1), (1, 0), (0, 1) \Big\}.
    \end{equation*}

\end{prop}

\begin{proof}

    We divide the proof into three steps. 
    In Step one we prove the inequality \eqref{eq:2.asymarea}. 
    The equality case is characterized in the second step under the assumption $e_2 \in \D(K)$, and in the third step in the general case.

    \vspace{5 pt}
    \noindent
    \textit{Step one}:
    by Minkowski's second theorem and Eggleston’s inequality we have
    \begin{equation*}
         6 \le \abs{K} \abs{(K_\sym)^\circ} \le \frac{4 \abs{K}}{\lambda_1((K_\sym)^\circ) \lambda_2((K_\sym)^\circ)},
    \end{equation*}
    which implies that
    \begin{equation*}
        \abs{K} \ge \frac{3}{2} \lambda_1((K_\sym)^\circ) \lambda_2((K_\sym)^\circ) \ge \frac{3}{2} \lambda_1((K_\sym)^\circ)^2 \xupref{eq:2.firstwidth}{=} \frac{3}{8} \, w(K)^2.
    \end{equation*}
    If the equality holds, then $K$ must be a triangle by the equality condition of Eggleston's inequality.
    
    \vspace{5 pt}
    \noindent
    \textit{Step two}:
    Assume now that $K$ satisfies the equality in \eqref{eq:2.asymarea}; as observed in the previous step, $K$ is a triangle.
    Assume further that $e_2 \in \D(K)$, so that $w(K) = \phi(e_2) + \phi(- e_2)$, where $\phi = \phi_K$ denotes the support function of $K$.
    By possibly translating $K$, we may assume that the vertices of $K$ are given by $\pm (a, b), (c, d) \in \R^2$,
    with $\phi(\pm e_2) = (\pm a, \pm b) \cdot (\pm e_2) = b > 0$, $\abs{d} \le b$, and
    \begin{equation*}
        \phi(x_1, x_2) = \max \big\{ \abs{a x_1 + b x_2}, c x_1 + d x_2 \big\}, \qquad \qquad (x_1, x_2) \in \R^2.
    \end{equation*}
    Notice that $K$ might no longer contain the origin in its interior, but this property is not relevant for the inequality \eqref{eq:2.asymarea}.
    Let $m \in \Z$ and $r \in [0, 1)$ be such that 
    \begin{equation*}
        m b \le a < (m + 1) b, \qquad r \defeq \frac{a}{b} - m \in [0, 1).
    \end{equation*}
    By the assumption $e_2 \in \D(K)$ we have the bounds
    \begin{equation} \label{eq:2.proof_asym_1}
        \begin{split}
            2 b &= \phi(e_2) + \phi(- e_2) \le \phi(1, - m) + \phi(- 1, m) \\
            &= \max \big\{ \abs{a - m b}, c - m d \big\} + \max \big\{ \abs{a - m b}, - c + m d \big\} \\
            &= \abs{a - m b} + \max \big\{ \abs{a - m b}, \abs{c - m d} \big\}
            = r b + \max \big\{ r b, \abs{c - m d} \big\} = r b + \abs{c - m d},
        \end{split}
    \end{equation}
    where in the last passage the maximum cannot be achieved by $r b$, or else the inequality would be violated since $r < 1$. 
    Similarly,
    \begin{equation} \label{eq:2.proof_asym_2}
        \begin{split}
            2 b &= \phi(e_2) + \phi(- e_2) \le \phi(1, - m - 1) + \phi(- 1, m + 1) \\
            &= \max \big\{ \abs{a - (m + 1) b}, c - (m + 1) d \big\} + \max \big\{ \abs{a - (m + 1) b}, - c + (m + 1) d \big\} \\
            &= (1 - r) b + \max \big\{ (1 - r) b, \abs{c - (m + 1) d} \big\}.
        \end{split}
    \end{equation}
    Notice that if $\abs{c - (m + 1) d} < (1 - r) b$, then \eqref{eq:2.proof_asym_2} yields $r = 0$, and, in turn, $\abs{c - (m + 1) d} < b$. 
    On the other hand, by \eqref{eq:2.proof_asym_1} we also have $\abs{c - m d} \ge 2 b$, and these two inequalities are incompatible in view of the geometric constraint $\abs{d} \le b$.
    Therefore necessarily $\abs{c - (m + 1) d} \ge (1 - r) b$.

    Summing up, we can rewrite \eqref{eq:2.proof_asym_1} and \eqref{eq:2.proof_asym_2} as
    \begin{equation} \label{eq:2.proof_asym_3}
        \abs{c - m d} \ge (2 - r) b, \qquad \abs{c - (m + 1) d} \ge (1 + r) b.
    \end{equation}
    We remark that the two quantities $(c - m d)$ and $(c - (m + 1) d)$ must have the same sign, or else \eqref{eq:2.proof_asym_3} would contradict the constraint $\abs{d} \le b$.

    As $K$ attains the equality in \eqref{eq:2.asymarea}, we have
    \begin{equation*} \label{eq:2.proof_asym_4}
    \begin{split}
        \frac{3}{2} \, b^2 &= \frac{3}{8} \, w(K)^2 = \abs{K} = \abs{a d - b c}  = \abs{r b (c - (m + 1) d) + (1 - r) b (c - m d)} \\
        & = r b \abs{ c - (m + 1) d} + (1 - r) b \abs{c - m d} \xupref{eq:2.proof_asym_3}{\ge} r (1 + r) b^2 + (1 - r) (2 - r) b^2 = 2 (r^2 - r + 1) b^2,
    \end{split}
    \end{equation*}
    where the fifth equality follows from the fact that $(c - m d)$ and $(c - (m + 1) d)$ have the same sign.
    We have $r^2 - r + 1\ge \frac{3}{4}$ with equality for $r = \frac{1}{2}$. 
    Hence necessarily $r = \frac{1}{2}$ and the inequalities in \eqref{eq:2.proof_asym_3} are equalities, 
    that is $\abs{c - m d} = \frac{3}{2} b = \abs{c - (m + 1) d}$, i.e. $d = 0$ and $\abs{c} = \frac{3}{2} b$.
    Up to a reflection (which is a unimodular transformation), we can assume $c = \frac{3}{2}b$.
   
    In conclusion, if equality holds in \eqref{eq:2.asymarea} and $e_2 \in \D(K)$, then up to a reflection, a translation, and a dilation $K$ is a triangle of the family
    \begin{equation*}
    T_m \defeq \conv \cb{(m, 1), (- m - 1, - 1), \rb{ 1, 0}}, \qquad \qquad m \in \Z.
    \end{equation*}
    Finally, notice that any triangle $T_m$ can be transformed into $T_0=T$ by applying the unimodular matrix
    $U = \big( \begin{smallmatrix} 1 & - m \\ 0 & 1 \end{smallmatrix} \big)$.
    
    \vspace{5 pt}
    \noindent
    \textit{Step three}:
    we deduce the general case from Step two by the same argument as in the proof of \cref{prop:2.symarea}.
    Given $(p_0, q_0) \in \D(K)$, using Bézout's lemma we find $m, n \in \Z$ with $p_0 m + q_0 n = 1$,
    and we consider the unimodular matrix $U$ in \eqref{eq:2.proof_sym_1}.
    Since both sides of \eqref{eq:2.asymarea} are invariant under unimodular transformations, and $e_2 \in \D(U K)$, 
    by the previous step if equality holds in \eqref{eq:2.asymarea} then $K$ is equal (up to translation, rescaling, and unimodular transformation) to $U^{- 1} T$.
    Conversely, it is easily checked that the triangle $T$ satisfies the equality in \eqref{eq:2.asymarea}.
    \qedhere
    
\end{proof}

For a general lattice we have the following two corollaries, which follow from \cref{prop:2.symarea}, \cref{prop:2.asymarea}, and the identities in \cref{rem:2.latticetrans}.

\begin{cor} \label{cor:2.symarea_latt}

    Let $K \subset \R^2$ be a symmetric convex body, and $\Lambda = A_\Lambda \Z^2 \subset \R^2$ be a lattice for some $A_\Lambda \in \GL(\R^2)$. Then,
    \begin{equation*}
        \abs{K} \ge d(\Lambda) \frac{w_\Lambda(K)^2}{2},
    \end{equation*}
    and equality holds if and only if, up to dilations, $K$ is equal to a parallelogram $A_\Lambda U P_s$ for some unimodular matrix $U \in \GL(\Z^2)$ and $s \in [0, 1)$, 
    where $P_s$ is defined as in \cref{prop:2.symarea}.

\end{cor}

\begin{cor}

    Let $K \subset \R^2$ be any convex body, and $\Lambda = A_\Lambda \Z^2 \subset \R^2$ be a lattice for some $A_\Lambda \in \GL(\R^2)$. Then,
    \begin{equation*}
        \abs{K} \ge \frac{3}{8} \, d(\Lambda) w_\Lambda(K)^2,
    \end{equation*}
    and equality holds if and only if, up to dilations and translations, $K$ is equal to a triangle $A_\Lambda U T$ for some unimodular matrix $U \in \GL(\Z^2)$,
    where $T$ is defined as in \cref{prop:2.asymarea}.

\end{cor}

We conclude this section by establishing another lower bound on the area of a convex body $K$ in terms of its lattice width 
and the first successive minimum of its symmetrized body.

\begin{prop} \label{prop:2.torusemb}

    Let $K \subset \R^2$ be any convex body. Then
    \begin{equation} \label{eq:2.torusemb}
        \abs{K} \ge \frac{w(K)}{\lambda_1(K_\sym)}.
    \end{equation}
    Equality holds in \eqref{eq:2.torusemb} if and only if either $K$ is a triangle, 
    or $K$ is a quadrilateral having a diagonal parallel to some $z \in \Z^2 \setminus \cb{0}$ 
    realizing $\lambda_1(K_\sym)$, with $z^\perp \in \D(K)$.
    
\end{prop}

\begin{proof}

    By definition of $\lambda_1(K_\sym)$, there exists $(p, q) \in \lambda_1(K_\sym) K_\sym \cap \Z^2 \setminus \cb{0}$.
    By symmetry of $K_\sym$, we also have $(- p, - q) \in \lambda_1(K_\sym) K_\sym$.
    By noticing that both sides of the inequality \eqref{eq:2.torusemb} are translation invariant, by possibly translating $K$ we can assume that $\pm (p, q) \in \lambda_1(K_\sym) K$,
    or equivalently, $\frac{\pm (p, q)}{\lambda_1(K_\sym)} \in  K$.
    Now, considering the orthogonal vectors $\pm (- q, p)$, 
    there exist $x, y \in K$ with $\phi(- q, p) = x \cdot (- q, p)$ and $\phi(q, - p) = y \cdot (q, - p)$, so that
    \begin{equation*}
        (x - y) \cdot (-q, p) = \phi(- q, p) + \phi(q, - p) \ge w(K).
    \end{equation*}
    We can estimate the area of $K$ from below by the area of the polygon $Q \defeq \conv \cb{\frac{\pm (p, q)}{\lambda_1(K_\sym)}, x, y}$, which is contained in $K$ by convexity. 
    $Q$ can be decomposed into two triangles sharing the base segment in the direction $(p, q)$ joining $\frac{\pm (p, q)}{\lambda_1(K_\sym)}$ and having third vertices $x$ and $y$ 
    (where one of the triangles might be degenerate); the sum of their heights is exactly the length of the projection of $x - y$ in the direction orthogonal to $(p, q)$ (see \cref{fig:2.torusemb}).

    \begin{figure}[h]
        \centering
        \begin{tikzpicture} [scale = 2]
            \draw[fill = darkgreen!35, rotate = 45] ({sqrt(2)}, 0) arc(0:180:{sqrt(2)} and 0.5) -- (- 1, - 0.25) -- (0, - 0.5) -- (0.5, - 0.5) -- cycle;
            \draw[fill = darkgreen, rotate = 45] ({sqrt(2)}, 0) -- (0, 0.5) -- (- {sqrt(2)}, 0) -- (0.5, - 0.5) -- cycle;
            \draw[fill = black] (1, 1) circle(0.02) node[anchor = south west]{\scriptsize{$\frac{(p, q)}{\lambda_1(K_\sym)}$}};
            \draw[fill = black] (- 1, - 1) circle(0.02) node[right]{\scriptsize{$\frac{(- p, - q)}{\lambda_1(K_\sym)}$}};
            \draw[dashed, fill = black, rotate = 45] (0, 0) -- (0, 0.5) circle(0.02) node[anchor = south east]{\scriptsize{$x$}};
            \draw[dashed, fill = black, rotate = 45] (0.5, 0) -- (0.5, - 0.5) circle(0.02) node[anchor = north west]{\scriptsize{$y$}};
            \draw(0.5, 1.1) node{$K$};
            \draw(0.75, 0.5) node{$Q$};
            \draw[dashed] (- 1, - 1) -- (1, 1);
            \draw[thin, rotate = 45] (- 1.6, - 0.5) -- (- 1.6, 0.5);
            \draw[thin, rotate = 45] (- 1.65, - 0.5) -- (- 1.55, - 0.5);
            \draw[thin, rotate = 45] (- 1.65, 0.5) -- (- 1.55, 0.5);
            \draw[thin, rotate = 45] (- 1.75, 0) node{\rotatebox{- 45}{\scriptsize{$\frac{\phi(q, - p) + \phi(- q, p)}{\sqrt{p^2 + q^2}}$}}};

            \draw[-latex, thick] (- 1.5, 0) -- (1.5, 0) node[below]{$x_1$};
            \draw[-latex, thick] (0, - 1.5) -- (0, 1.5) node[left]{$x_2$};
        \end{tikzpicture}
        
        \caption{The construction of the polygon $Q$ in the proof of \cref{prop:2.torusemb}.}
        \label{fig:2.torusemb}
    \end{figure}
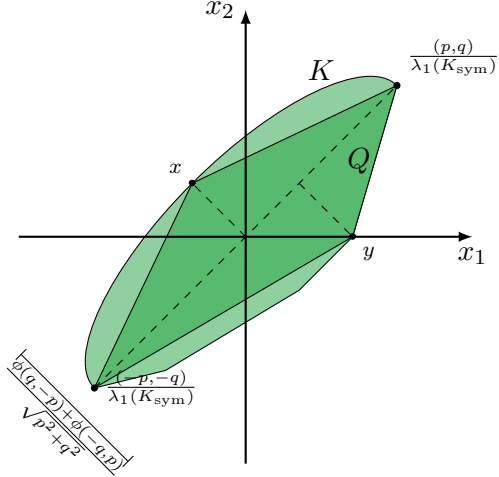

    Thus we find
    \begin{equation*}
        \abs{K} \ge \abs{Q} \ge \frac{1}{2} \rb{\frac{2 \sqrt{p^2 + q^2}}{\lambda_1(K_\sym)}} \rb{\frac{w(K)}{\sqrt{p^2 + q^2}}} = \frac{w(K)}{\lambda_1(K_\sym)},
    \end{equation*}
    proving \eqref{eq:2.torusemb}.

    If \eqref{eq:2.torusemb} holds with equality, $K$ coincides (up to a translation) with the polygon $Q$, and $\phi(- q, p) + \phi(q, - p) = w(K)$. 
    Therefore the conditions in the statement are necessary for the equality case.
    
    Conversely, equality in \eqref{eq:2.torusemb} holds for any quadrilateral satisfying the conditions in the statement, by the same computation as in the first part of the proof. 
    Finally, we show that any triangle achieves equality in \eqref{eq:2.torusemb}.
    Denote by $x$, $y$ and $a$ its vertices.
    Given $z \in \D(K)$, up to relabeling, we may assume that $\phi(z) = x \cdot z$ and $\phi(- z) = - y \cdot z$, so that $w(K) = (x - y) \cdot z$.
    Consider the section of $K$ with the line orthogonal to $z$ passing through $a$, whose length we denote by $\ell > 0$ (note that positivity of $\ell$ follows from the inequalities $y \cdot z \le a \cdot z \le x \cdot z$).
    In particular $b \defeq a + \ell\frac{z^\perp}{\abs{z}}\in K$ (up to a sign), hence $z^\perp = \frac{\abs{z}}{\ell}(b-a)\in\frac{2\abs{z}}{\ell} K_\sym \cap \Z^2 \setminus \cb{0}$, so $\lambda_1(K_\sym) \le \frac{2\abs{z}}{\ell}$. 
    Therefore the area of $K$ is
    \begin{equation*}
        \abs{K} = \frac{\ell}{2}(x - y) \cdot \frac{z}{\abs{z}} = \frac{w(K)\ell}{2\abs{z}} \le \frac{w(K)}{\lambda_1(K_\sym)},
    \end{equation*}
    which is the opposite inequality in \eqref{eq:2.torusemb}.
    \qedhere
    
\end{proof}

We remark that the inequality \eqref{eq:2.torusemb} can be equivalently formulated in terms of the ratio between the first successive minima of the symmetrized body of $K$ and its polar: 
$\frac{\lambda_1((K_\sym)^\circ)}{\lambda_1(K_\sym)} \le \frac{\abs{K}}{2}$.
As before, for a general lattice we have the following extension of the previous result.

\begin{cor}

    Let $K \subset \R^2$ be any convex body and $\Lambda \subset \R^2$ be a lattice. Then it holds
    \begin{equation} \label{eq:2.torusemb_latt}
        \abs{K} \ge d(\Lambda) \frac{w_\Lambda(K)}{\lambda_1(K_\sym, \Lambda)}.
    \end{equation}
    Equality holds in \eqref{eq:2.torusemb_latt} if and only if either $K$ is a triangle, 
    or $K$ is a quadrilateral with a diagonal parallel to some $z \in \Lambda \setminus \cb{0}$ 
    realizing $\lambda_1(K_\sym, \Lambda)$, with $z^\perp/d(\Lambda) \in \D_\Lambda(K)$.
    
\end{cor}

\section{Proof of \texorpdfstring{\cref*{thm:main}}{Theorem \getrefnumber{thm:main}}} \label{sec:3}

\subsection{Lamellar configurations on the torus}

In the following, we shall consider lamellar configurations with directions different from the horizontal and vertical ones, 
and for that purpose we compute the length of a line in the torus, defined as the quotient of a line in $\R^2$.
Notice that a line $r \subset \T^2$ has finite length if and only if it is closed, 
meaning that there exist two different points on the line whose difference belongs to $\Z^2$.

\begin{lem} \label{lem:3.linelen}

    Let $r \subset \T^2$ be a line in the flat torus $\T^2$ with normal vector $\nu = (\nu_1, \nu_2) \in \S^1$.
    Then, $r$ has finite length if and only if there exist $p, q \in \Z$ with $\GCD (p, q) = 1$ such that
    \begin{equation*}
        \nu_1 = \frac{p}{\sqrt{p^2 + q^2}}, \qquad \qquad \nu_2 = \frac{q}{\sqrt{p^2 + q^2}}.
    \end{equation*}
    In that case, the length of $r$ is given by
    \begin{equation*}
        \Ha^1(r) = \sqrt{p^2 + q^2}.
    \end{equation*}
    
\end{lem}

\begin{proof}

    Up to a translation, we may assume that $r$ passes through the point $(0, 0)$ of the torus, 
    so that $r = \cb{t \tau \in \T^2 : t \in \R}$, where $\tau \defeq (- \nu_2, \nu_1) \in \S^1$ is the vector giving the direction of $r$.
    Then $r$ has finite length if and only if there exists $t > 0$ with $t \tau \in \Z^2$, that is, $t \nu_1, t \nu_2 \in \Z$.
    Let $t_0 \defeq \min \cb{t > 0 : t \tau \in \Z^2}$, and set $p \defeq t_0 \nu_1 \in \Z$, $q \defeq t_0 \nu_2 \in \Z$.
    Notice that $\GCD (p, q) = 1$, otherwise the minimality of $t_0$ would be contradicted.
    As a consequence,
    \begin{equation*}
        \rb{\frac{p}{t_0}, \frac{q}{t_0}} = (\nu_1, \nu_2) \in \S^1, 
    \end{equation*}
    which in turn implies that $t_0 = \sqrt{p^2 + q^2}$.
    The converse implication follows easily.
    Finally, the length of $r$ is $\Ha^1(r) = \abs{t_0 \tau} = t_0 = \sqrt{p^2 + q^2}.$ 
    \qedhere

\end{proof}

By the previous lemma it makes sense to give the following definition.

\begin{defi} \label{defi:3.lamconf}

    Given $z \in \Z^2 \setminus \cb{0}$ with $\GCD(z_1, z_2) = 1$, the lamella $L_M^z$ with area $M \in (0, 1)$ and orthogonal direction $z$ is defined as
    \begin{equation*}
        L_M^z \defeq \cb{t z^\perp + s z \in \T^2 : t \in \R, \ 0 \le s \le \frac{M}{\abs{z}^2}}.
    \end{equation*}
    
\end{defi}

The area of $L_M^z$ is $M$, since its height is equal to $\frac{M}{\abs{z}}$, and by the previous lemma the base has length $\abs{z}$.
Moreover, given a surface tension $\phi$ associated with a convex body $K$, in view of \cref{lem:3.linelen}, 
the anisotropic perimeter of the lamella $L_M^z$ is simply $P_{\T^2}^\phi(L_M^z) = \phi(z) + \phi(- z)$ (which is independent of its area $M$). 
Therefore, the set $\D(K)$ of minimizing lattice directions for $K$ dictates the optimal orientation of the lamellae, 
and the minimal perimeter coincides with the minimal lattice width $w(K)$ of $K$ (see \cref{defi:2.lattice_width}):
\begin{equation} \label{eq:3.per_lamella}
    \min_{z \in \Z^2 \setminus\cb{0}} P_{\T^2}^\phi (L_M^z) = P_{\T^2}^\phi (L_M^{\bar{z}}) = w(K) \qquad\text{for all } \bar{z} \in \D(K).
\end{equation}

\subsection{The uniformly elliptic case}

In this subsection we prove \cref{thm:main} for the class of uniformly elliptic surface tensions, 
whose Wulff shape is a uniformly convex set with boundary of class $C^2$ (see \cite{Sch13}). 
In this setting we can make use of standard regularity results for minimizers of the anisotropic perimeter.

\begin{defi} \label{defi:2.uniell}

    A surface tension $\phi : \R^2 \to [0, + \infty)$ 
    is said to be \textit{uniformly elliptic with constant $\lambda > 0$} (or simply \textit{$\lambda$-elliptic}) 
    if $\phi\in C^2(\R^2 \setminus \cb{0})$ and
    \begin{equation*}
        \nabla^2 \phi(x) \sb{x^\perp, x^\perp} \ge \lambda \qquad \qquad \text{for any } x \in \S^1. 
    \end{equation*}
    
\end{defi}

We can now prove \cref{thm:main}.

\begin{proof}[Proof of \cref{thm:main} assuming $\phi$ uniformly elliptic]

    We divide the proof into three steps. In the first step we collect some preliminary remarks on the candidate minimizers for \eqref{eq:1.wulffper}. 
    In the second step we show that any minimizer $E$ has boundary of class $C^2$ and its curvature has constant sign. 
    Finally, in the third step we classify the possible solutions of \eqref{eq:1.wulffper} in terms of the sign of the curvature of the boundary.
    Recall that we are given a uniformly elliptic surface tension $\phi$ whose associated Wulff shape is denoted by $K$.
    
    \vspace{5 pt}
    \noindent
    \textit{Step one}:
    we begin with some preliminary observations on the candidate minimizers for problem \eqref{eq:1.wulffper}.
    We define the quantity
    \begin{equation*}
        \overline \mu \defeq \inf \cb{\mu > 0 : (\mu K) \cap (\mu K + z) \ne \emptyset \text{ for some } z \in \Z^2 \setminus \cb{0}},
    \end{equation*}    
    which corresponds to the critical scaling factor for $K$ to embed isometrically in the torus; in other words, for $\mu\le\overline{\mu}$, 
    the interiors of the translated copies $\mu K + z$ for $z\in\Z^2$ are mutually disjoint, 
    so the rescaled Wulff shape $\mu K$ does not self-intersect (or only self-intersects at its boundary for $\mu=\overline{\mu}$) when wrapped around the torus.
    We observe that
    \begin{align*}
        \overline{\mu} &= \inf \cb{\lambda > 0 : \text{ there exist } x, y \in \lambda K \text{ with } x - y \in \Z^2 \setminus \cb{0}} \\
        &= \inf \cb{\lambda > 0 : 2\lambda K_\sym \cap \Z^2 \ne \cb{0}} 
        = \frac{1}{2} \, \lambda_1(K_\sym),
    \end{align*}
    and therefore by \cref{prop:2.torusemb} we have the lower bound
    \begin{equation} \label{prf:main-mubar}
         \overline \mu > \frac{w(K)}{2\abs{K}}.
    \end{equation}

    Next, for $M \in (0,1)$ we let
    \begin{equation} \label{prf:main-muM}
        \mu_M \defeq \sqrt{\frac{M}{\abs{K}}}
    \end{equation}
    be the scaling factor such that $\abs{\mu_M K} = M$. 
    By \eqref{eq:2.per-wulff} and \eqref{eq:3.per_lamella} the anisotropic perimeters of the rescaled Wulff shape $\mu_M K$ and of the lamella $L_M^z$, for $z \in \D(K)$, are
    \begin{equation} \label{prf:main-per}
        P^\phi(\mu_M K) = \mu_M P^\phi(K) = 2 \sqrt{M \abs{K}}, \qquad P_{\T^2}^\phi(L_M^{z}) = w(K),
    \end{equation}
    respectively (recall that $P^\phi$ denotes the anisotropic perimeter in $\R^2$). 
    Hence the quantity
    \begin{equation*}
        \alpha(K) \defeq \frac{w(K)^2}{4 \abs{K}}
    \end{equation*}
    corresponds to the critical area for which $P^\phi(\mu_{\alpha(K)} K) = P_{\T^2}^\phi(L_{\alpha(K)}^{z})$. 
    We remark that, by \cref{prop:2.symarea} and \cref{prop:2.asymarea}, if $K$ is symmetric then we always have $\alpha(K) \le \frac{1}{2}$, 
    whereas $\alpha(K)\le\frac{2}{3}$ in the general case of a non-symmetric convex body $K$, where it may happen that $\alpha(K)$ exceeds $\frac{1}{2}$. 
    It is important to observe that, by \eqref{prf:main-mubar},
    \begin{equation} \label{prf:main-mubar2}
        \mu_{\alpha(K)} < \overline{\mu},
    \end{equation}
    meaning that, for all $M \le \alpha(K)$, the rescaled Wulff shape $\mu_M K$ with area $M$ is embedded in the torus $\T^2$ without self-intersections.

    Finally, we also compute the anisotropic perimeter of the complement of $- K$. Denoting by $\tilde \phi : \R^2 \to [0, + \infty)$ 
    the surface tension given by $\tilde \phi(x) \defeq \phi(- x)$, $x \in \R^2$, whose associated Wulff shape is $W_{\tilde \phi} = - K$, 
    we see that for any set of finite perimeter $F \subset \R^2$ the complement $F^c = \R^2 \setminus F$ is such that $P^{\tilde \phi} (F^c) = P^\phi (F)$. 
    Therefore
    \begin{equation} \label{prf:main-per2}
        P^\phi \big( (- \mu_{1 - M} K)^c \big) = P^{\tilde \phi} \big( - \mu_{1 - M} K \big)
        = \mu_{1 - M} P^{\tilde \phi} (W_{\tilde \phi}) = 2 \sqrt{(1 - M) \abs{K}}. 
    \end{equation}

    \vspace{5 pt}
    \noindent
    \textit{Step two}:
    let $E$ be a minimizer for problem \eqref{eq:1.wulffper}.
    By standard regularity results for planar minimizers of the anisotropic perimeter in the case of uniformly elliptic surface tensions, 
    the boundary of $E$ is of class $C^1$ (see for instance \cite[Theorem~6.4; Theorem~6.18]{ANP02}). 
    Next, by a first variation argument we improve the regularity to $C^2$ and we show that the Euclidean mean curvature of $\partial E$ has constant sign.

    In local coordinates, given any point $x_0 \in \partial E$ we can express $\partial E$ in a neighborhood of $x_0$ as the graph of a $C^1$-function: 
    in a suitable coordinate system,
    \begin{equation*}
        \big( \partial E - x_0 \big) \cap (- r, r)^2 = \Big\{ (s, u(s)) : \abs{s} < r \Big\}, \quad
        \big( E - x_0 \big) \cap (- r, r)^2 = \Big\{ (s, t) \in (- r, r)^2 : t < u(s) \Big\},
    \end{equation*}
    for some function $u \in C^1((- r, r))$ with $u(0) = 0$, and $r > 0$ small enough.
    By computing the first variation of the anisotropic perimeter along a variation of the form $u + \varepsilon h$, 
    for $h \in C^1_{c}((- r, r))$ with $\int_{- r}^r h = 0$, we find that there exists a constant $H_0 \in \R$ such that $u$ is a weak solution to
    \begin{equation} \label{prf:weaksol}
        \big( \partial_1 \phi(- u'(s), 1) \big)' = H_0 \qquad\text{for } s \in (- r,r),
    \end{equation}
    and \cite[Theorem 7.56]{AFP00} ensures that $u \in W^{2, 2}_\loc((- r, r))$.
    As a consequence,
    \begin{equation} \label{prf:main-curv-2}
        u'' = - \frac{H_0}{\partial_{1 1}^2 \phi(- u', 1)} \qquad\text{in }(-r,r),
    \end{equation}
    where $\partial_{1 1}^2 \phi(- u', 1) > 0$ is uniformly bounded away from zero by uniform ellipticity of $\phi$. 
    This proves that $u$ is of class $C^2$ and, in turn, any minimizer of \eqref{eq:1.wulffper} has boundary of class $C^2$.

    Furthermore, the Euclidean mean curvature $H_E(x) \defeq \dive_\tau \nu_E(x)$ of $\partial E$ with respect to the outer normal is represented locally as
    \begin{equation*}
        H_E(s,u(s)) = -\frac{u''(s)}{(1 + u'(s)^2)^{3/2}} \qquad \text{for } s \in (- r,r).
    \end{equation*}
    Hence from \eqref{prf:main-curv-2} it follows that $H_E$ has constant sign on $\partial E$, coinciding with the sign of $H_0$.
    
    \vspace{5 pt}
    \noindent
    \textit{Step three}:
    We consider a minimizer $E$ for problem \eqref{eq:1.wulffper} and its periodic extension $E_{\R^2}$ to $\R^2$, 
    and we study its geometry in relation to the sign of $H_0$ in \eqref{prf:weaksol}.
    Notice that $E$ is connected as a subset of $\T^2$ 
    (otherwise, translating one connected component so that it touches another would either reduce the perimeter, or create a singularity, contradicting minimality).
    
    \begin{itemize}
        \item 
            If $H_0 = 0$, then $H_E \equiv 0$ and $\partial E_{\R^2}$ is the union of parallel lines.
            Since $E$ is connected, it must be $E = L_M^z$ for some $z \in \Z^2 \setminus \cb{0}$.
            By \eqref{eq:3.per_lamella} we conclude that $z \in \D(K)$ and $P_{\T^2}^\phi(E) = w(K)$.
        
        \item 
            If $H_0 > 0$, fix one of the connected components of $E_{\R^2}$, denoted by $F \subset \R^2$, whose boundary has strictly positive curvature.
            Hence $F$ is strictly convex and bounded, 
            which in turn implies that every connected component of $E_{\R^2}$ is bounded and strictly convex.
            Since $E$ is connected as a subset of $\T^2$, it holds $E_{\R^2} = \bigcup_{z \in \Z^2}(F + z)$ 
            which in turn implies $\abs{F} = \abs{E} = M$. 
            By minimality of $\mu_M K$ for the anisotropic isoperimetric problem \eqref{eq:2.wulff} in $\R^2$ it follows that
            \begin{equation} \label{prf:main-est}
                P_{\T^2}^\phi(E) = P^\phi(F) \ge P^\phi (\mu_M K),
            \end{equation}
            with equality if and only if $F = \mu_M K$ (up to translation).
            
            Consider first the case $M \le \alpha(K) \wedge \frac{1}{2}$.
            As observed in Step~one, as a consequence of the inequality \eqref{prf:main-mubar2} 
            the rescaled Wulff shape $\mu_M K$ is embedded in the torus $\T^2$ without self-intersections, so that we can continue in \eqref{prf:main-est} as
            \begin{equation*}
                 P_{\T^2}^\phi(E) = P^\phi(F) \ge P^\phi (\mu_M K) = P_{\T^2}^\phi (\mu_M K) \ge P_{\T^2}^\phi(E).
            \end{equation*}
            Therefore the inequalities are actually equalities and $E = \mu_M K$.

            Let us show that the complementary case $M > \alpha(K) \wedge \frac{1}{2}$ is incompatible with the condition $H_0 > 0$.
            If $\alpha(K) \le 1/2$, we have $M > \alpha(K)$ and hence
            \begin{equation*}
                P_{\T^2}^\phi(E) \xupref{prf:main-est}{\ge} P^\phi (\mu_M K) > P^\phi (\mu_{\alpha(K)} K)
                = \frac{w(K)}{2 \abs{K}} P^\phi (K) = w(K) = P_{\T^2}^\phi (L_M^z)
            \end{equation*}
            for $z \in \D(K)$, which contradicts the minimality of $E$.
            
            If instead $\alpha(K) > 1/2$, we have that $1 - M < 1/2 < M$, and, denoting $\tilde \phi(x) \defeq \phi(- x)$ as in Step~one, one gets
            \begin{align*}
                P_{\T^2}^\phi(E)
                & \xupref{prf:main-est}{\ge} P^\phi (\mu_M K)
                \xupref{prf:main-per}{=} 2 \sqrt{M \abs{K}}
                > 2 \sqrt{(1 - M) \abs{K}} \\
                & \xupref{prf:main-per2}{=} P^\phi \big( (- \mu_{1 - M} K)^c \big)
                = P_{\T^2}^\phi \big( (- \mu_{1 - M} K)^c \big),
            \end{align*}
            where the last equality follows since the set $\mu_{1 - M} K$ is embedded in the torus $\T^2$, in view of the inequality \eqref{prf:main-mubar2} 
            and the assumption $\alpha(K) \ge 1/2 > 1 - M$. Since $\abs{(- \mu_{1 - M} K)^c} = M$, this again contradicts the minimality of $E$.

            In conclusion, we showed that, if $H_0 > 0$, then necessarily $E$ coincides with $\mu_M K$ and $M\leq\alpha(K) \wedge \frac{1}{2}$.
          
        \item 
            If $H_0 < 0$, then $H_E < 0$, and the complement $E^c = \T^2 \setminus E$
            is a minimizer of \eqref{eq:1.wulffper} for the reflected surface tension $\tilde \phi(x) = \phi(- x)$ with $\abs{E^c} = 1 - M$ and $H_{E^c} > 0$.
            We conclude by the previous point that necessarily $1 - M \le \alpha(K) \wedge \frac{1}{2}$, that is, 
            $M \ge 1 - (\alpha(K) \wedge \frac{1}{2})$, and $E^c = \mu_{1 - M} W_{\tilde \phi} = - \mu_{1 - M} K$.
    
    \end{itemize}
    In conclusion, the only possible minimizers for \eqref{eq:1.wulffper} are $\mu_M K$ (for $M \le \alpha(K) \wedge \frac{1}{2}$), 
    $(- \mu_{1 - M} K)^c$ (for $M \ge 1 - (\alpha(K) \wedge \frac{1}{2})$), and $L^z_M$ for $z \in \D(K)$. 
    By comparing the respective anisotropic perimeters, see \eqref{prf:main-per} and \eqref{prf:main-per2}, we obtain the conclusion.
    \qedhere
   
\end{proof}

\subsection{The general case}

We now provide the full proof of \cref{thm:main} for a general anisotropy $\phi$, by approximating $\phi$ with uniformly elliptic surface tensions.

\begin{proof}[Proof of \cref{thm:main}]

    Let $\phi$ be a surface tension with Wulff shape $K$, as in the statement.
    We construct a sequence $\cb{\phi_h}_{h \in \N}$ of uniformly elliptic surface tensions
    such that $\phi \le \phi_h$ for $h \in \N$, and $\lim_{h \to \infty} \phi_h = \phi$ uniformly on $\S^1$.
    
    To this end, we first use the approximation argument in \cite[Theorem~3.3.1]{Sch13} to find a sequence $\{ \tilde{\phi}_h \}_{h \in \N}$ of surface tensions 
    of class $C^\infty (\R^2 \setminus \cb{0})$ such that $\tilde{\phi}_h \to \phi$ uniformly on $\S^1$ as $h \to \infty$ 
    (equivalently, the corresponding Wulff shapes $\tilde{K}_h$ converge to $K$ in the Hausdorff metric). 
    Then, setting $\delta_h \defeq \max_{x \in \S^1} \abs{\phi(x) - \tilde{\phi}_h(x)} \to 0$, we define
    \begin{equation*}
        \phi_h(x) \defeq \tilde{\phi}_h(x) + \rb{\delta_h + \frac{1}{h}} \abs{x}, \qquad x \in \R^2
    \end{equation*}
    (corresponding to the Wulff shape $K_h = \tilde{K}_h + B_{\delta_h + 1/h}$), so that $\phi_h \in C^\infty (\R^2 \setminus \cb{0})$ is uniformly elliptic for any $h \in \N$,
    $\phi_h \ge \phi$, and $\phi_h \to \phi$ as $h \to \infty$ uniformly on compact sets.
    This gives the desired approximation.

    Now, having fixed $M \in (0, 1)$, let $E_h \subset \T^2$ be a minimizer for
    \begin{equation} \label{eq:3.wulffper_h}
    \min \cb{P_{\T^2}^{\phi_h}(E) : E \subseteq \T^2, \ \abs{E} = M}.
    \end{equation}
    Applying the compactness theorem for sets of finite perimeter, 
    up to a subsequence $E_h \to E$ (in the sense that $\abs{E_h \Delta E}\to 0$) for some set $E \subset \T^2$ with $\abs{E} = M$.
    In particular, $E$ is a minimizer in \eqref{eq:1.wulffper} for the surface tension $\phi$: 
    indeed, given any set $F \subset \T^2$ with finite perimeter and $\abs{F} = M$, we have
    \begin{equation*}
        P_{\T^2}^\phi(E) \le \liminf_{h \to \infty} P_{\T^2}^\phi (E_h) 
        \le \liminf_{h \to \infty} P_{\T^2}^{\phi_h}(E_h)
        \le \liminf_{h \to \infty} P_{\T^2}^{\phi_h}(F) = P_{\T^2}^\phi(F),
    \end{equation*}
    where we used the lower semicontinuity of the anisotropic perimeter.
    Hence any convergent subsequence of minimizers $\cb{E_h}_{h \in \N}$ of \eqref{eq:3.wulffper_h} converges to a minimizer of \eqref{eq:1.wulffper}.
    Moreover, $w(K_h) \to w(K)$ and $\alpha(K_h) \to \alpha(K)$ as $h \to \infty$.
    
    If $M < \alpha(K) \wedge \frac{1}{2}$ then also $M < \alpha(K_h) \wedge \frac{1}{2}$ for $h$ sufficiently large,
    so that by \cref{thm:main} in the uniformly elliptic case, proved in the previous subsection, we necessarily have $E_h = \sqrt{\frac{M}{\abs{K_h}}} K_h$ (up to a translation).
    Hence $\sqrt{\frac{M}{\abs{K}}} K$ is a minimizer for $M < \alpha(K) \wedge \frac{1}{2}$.
    In a similar way we obtain a minimizer in the case $M > 1 - (\alpha(K) \wedge \frac{1}{2})$.

    For the lamellar phase, let $\alpha(K) < M < 1 - \alpha(K)$, so that, as before, $L^{z_h}_M$ for $z_h\in\D(K_h)$ is a minimizer of \eqref{eq:3.wulffper_h} for $h$ large enough. 
    Since $\cb{z_h}_{h \in \N}$ is uniformly bounded, up to subsequences $z_h = \bar{z} \in \Z^2 \setminus \cb{0}$, and it is easily proved that $\bar{z} \in \D(K)$:
    hence $L^{\bar{z}}_M$ is a minimizer for \eqref{eq:1.wulffper} and, in turn, $L^{z}_M$ is also a minimizer for all $z \in \D(K)$.
    
    Finally, the transition cases, when $M$ coincides with one of the transition thresholds, follow by a continuity argument.
    \qedhere
    
\end{proof}

We emphasize that the previous argument only shows that the minimizers listed in \cref{thm:main} can be obtained as limits of minimizers for uniformly elliptic surface tensions. 
It does not imply that \emph{every} minimizer of the limiting problem must be one of these configurations: additional minimizers may exist (including lamellar configurations) 
which are not limits of minimizers of the approximating problems. 

\section{Additional remarks and examples} \label{sec:4}

We collect in this section several examples showing how the presence of anisotropy in the surface tension 
might significantly alter the qualitative properties of minimizers, in contrast with the isotropic case.

\begin{rem} \label{rem:4.lamdir}

    Given any lattice direction $(p_0, q_0) \in \Z^2 \setminus \cb{0}$ with $\GCD(p_0, q_0) = 1$, 
    it is possible to construct a convex body $K$ such that $(p_0, q_0) \in \D(K)$; 
    that is, for the surface tension associated with $K$, the preferred orientation for the lamellae is in the direction orthogonal to $(p_0, q_0)$ (see \eqref{eq:3.per_lamella}).
    
    Indeed, for $\varepsilon > 0$ (to be chosen later) we define the uniformly elliptic, symmetric surface tension
    \begin{equation*}
        \phi(x_1, x_2) = \sqrt{(x_1 q_0 - x_2 p_0)^2 + \varepsilon^2 (x_1 p_0 + x_2 q_0)^2}, \qquad (x_1, x_2) \in \R^2,
    \end{equation*}
    whose associated Wulff shape is an ellipse rotated by an angle depending on $(p_0, q_0)$ (see \cref{fig:4.lamconf}).
    We have $\phi(p_0, q_0) = \varepsilon (p_0^2 + q_0^2)$.
    For any $(p, q) \in \Z^2 \setminus \cb{0}$ with $\GCD(p, q) = 1$ and $(p, q) \ne (p_0, q_0)$, it holds $p q_0 \ne q p_0$ and
    \begin{equation*}
        \phi(p, q) = \sqrt{(p q_0 - q p_0)^2 + \varepsilon^2 (p p_0 + q q_0)^2} \ge \abs{p q_0 - q p_0} \ge 1.
    \end{equation*}
    By choosing $\varepsilon < \frac{1}{p_0^2 + q_0^2}$, recalling that $\phi$ is symmetric, 
    we therefore obtain that $\phi(p_0, q_0) < \phi(p, q)$ for all $(p, q) \in \Z^2 \setminus \cb{0, \pm (p_0, q_0)}$, that is,
    $(p_0, q_0)\in\D(K)$, where $K$ is the Wulff shape of $\phi$. 
    By \cref{thm:main} the lamella $L_M^{(p_0, q_0)}$ is a global minimizer of \eqref{eq:1.wulffper} for $M$ close to $1/2$.

\end{rem}

\begin{figure}[ht]
    \centering
    \begin{tikzpicture}[scale = 1.7]
        \begin{scope}[rotate = 78.69]
            \draw[fill = darkgreen] (0, 0) ellipse (1 and 0.037);
            \draw[->] (0, - 0.037) -- (0, - 0.487) node[above] {\scriptsize $(p_0, q_0)$};
            \draw (0.3, 0) node[anchor = east]{$K$};
        \end{scope}
    \end{tikzpicture}
    \hspace{1.5 cm}
    \begin{tikzpicture}[scale = 1.7]
        \fill[darkgreen!35] (- 1, - 1) rectangle (1, 1);
        \fill[darkgreen] (- 0.7, 1) -- (- 0.9, 1) -- (- 1, 0.5) -- (- 1, - 0.5) -- cycle;
        \fill[darkgreen] (- 0.9, - 1) -- (- 0.7, - 1) -- (- 0.3, 1) -- (- 0.5, 1) -- cycle;
        \fill[darkgreen] (- 0.5, - 1) -- (- 0.3, - 1) -- (0.1, 1) -- (- 0.1, 1) -- cycle;
        \fill[darkgreen] (- 0.1, - 1) -- (0.1, - 1) -- (0.5, 1) -- (0.3, 1) -- cycle;
        \fill[darkgreen] (0.3, - 1) -- (0.5, - 1) -- (0.9, 1) -- (0.7, 1) -- cycle;
        \fill[darkgreen] (0.7, - 1) -- (0.9, - 1) -- (1, - 0.5) -- (1, 0.5) -- cycle;
        \draw (- 1, 0.5) -- (- 0.9, 1);
        \draw (- 1, - 0.5) -- (- 0.7, 1);
        \draw (- 0.9, - 1) -- (- 0.5, 1);
        \draw (- 0.7, - 1) -- (- 0.3, 1);
        \draw (- 0.5, - 1) -- (- 0.1, 1);
        \draw (- 0.3, - 1) -- (0.1, 1);
        \draw (- 0.1, - 1) -- (0.3, 1);
        \draw (0.1, - 1) -- (0.5, 1);
        \draw (0.3, - 1) -- (0.7, 1);
        \draw (0.5, - 1) -- (0.9, 1);
        \draw (1, 0.5) -- (0.7, - 1);
        \draw (1, - 0.5) -- (0.9, - 1);
        \draw (- 1, - 1) -- (- 1, 1) -- (1, 1) -- (1, - 1) -- cycle;
        \draw[thick, ->](- 1, - 0.00001) -- (- 1, 0.00001);
        \draw[thick, ->](1, - 0.00001) -- (1, 0.00001);
        \draw[thick, ->](- 0.00001, - 1) -- (0.00001, - 1);
        \draw[thick, ->](- 0.00001, 1) -- (0.00001, 1);
        \draw[->] (- 0.2, - 0.5) -- ++(- 11.31 : 0.15);  
        \draw(1.3, 0.8) node{$\T^2$};
    \end{tikzpicture}
    \caption{The Wulff shape of the surface tension of \cref{rem:4.lamdir} is, up to a scaling factor, 
    a very eccentric ellipse rotated by an angle depending on $(p_0, q_0)$.}
    \label{fig:4.lamconf}
\end{figure}
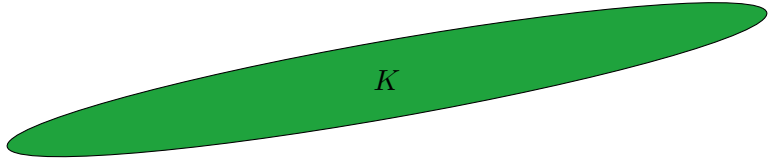

\begin{rem} \label{rem:4.lamskip}

    For a non-symmetric convex body $K$, the value $\alpha(K)$ in \eqref{eq:1.alphaK} might exceed $1/2$ 
    and in this case lamellar configurations are \emph{never} global minimizers; 
    the solution to \eqref{eq:1.wulffper} transitions directly from the droplet phase to the bubble phase at $M = 1/2$, skipping the lamellar phase.
    
    As an example, consider the surface tension $\phi : \R^2 \to [0, + \infty)$ defined as
    \begin{equation*}
        \phi(x_1, x_2) \defeq \max \cb{x_1, x_2, - x_1 - x_2}, \qquad \qquad (x_1, x_2) \in \R^2,
    \end{equation*}
    whose associated Wulff shape is the triangle $T$ achieving equality in \eqref{eq:2.asymarea}, depicted in \cref{fig:2.symeq}:
    in this case the quantity $\alpha(K)$ attains its possible maximum value $2/3$.
    
\end{rem}

\begin{rem} \label{rem:4.laminf}

    For arbitrary surface tensions there might be infinitely many minimizers of \eqref{eq:1.wulffper} for the same value of the area constraint $M$:
    consider a surface tension $\phi : \R^2 \to [0, + \infty)$, with $K \subset \R^2$ the corresponding Wulff shape, satisfying the two conditions in the statement of \cref{thm:uniqueness}.
    Since $\phi$ is positively one-homogeneous, the neighborhood of $z_0 \in \D(K)$ where $\phi$ is linear may be assumed to be an open cone $C$, 
    and we may find linearly independent unit vectors $\nu_1, \nu_2 \in C$ such that $a \nu_1 + b \nu_2 = z_0$ for some $a, b > 0$.
    In particular, $a, b$ and $\abs{z_0}$ are the side lengths of a triangle $\Delta_0$ oriented by $\nu_1, \nu_2$ and $- \frac{z_0}{\abs{z_0}}$, see \cref{fig:4.laminf}.
    Then, for $h \in \N$ large enough, we may attach $h$ copies of $\Delta_h = h^{- 1} \Delta_0$ to the side of the lamella $L_M^{z_0}$ with normal vector $z_0$, 
    obtaining a set denoted by $E_h$, see \cref{fig:4.laminf}.
    Eventually translating the two boundary components of $E_h$ along the direction of $z_0$, we may assume without loss of generality that $\abs{E_h} = M$.
    In this way, by the assumption of linearity on $C$, we immediately deduce
    \begin{equation*}
        P_{\T^2}^\phi (L_M^{z_0}) - P_{\T^2}^\phi(E_h) = \phi(z_0) - h \rb{\frac{a}{h} \phi(\nu_1) + \frac{b}{h} \phi(\nu_2)} 
        = \phi(z_0) - a \phi(\nu_1) - b \phi(\nu_2) = 0.
    \end{equation*}
    In this case every $E_h$ is a minimizer for the same value of $M$, proving the sufficiency condition stated in \cref{thm:uniqueness}.
    It is also worth noting that smooth minimizers may be constructed, by considering an appropriate sinusoidal curve in place of the side of the lamella $L_M^{z_0}$ with normal vector $z_0$.
    
\end{rem}

\begin{figure}[ht]
    \centering
    \begin{tikzpicture} [scale = 1.4]
        \draw[fill = darkgreen] (0, 0) -- (4, 0) -- (1, 1.5) -- cycle;
        \draw[->] (2, 0) node[anchor = south east]{$\abs{z_0}$} -- (2, - 1) node[anchor = south east]{$-\frac{z_0}{\abs{z_0}}$}; 
        \draw[->, rotate around = {153.4: (2.5, 0.75)}] (2.5, 0.75) node[xshift = - 0.6 cm]{$b$} -- (2.5, - 0.25) node[anchor = west]{$\nu_2$};
        \draw[->, rotate around = {- 116.6: (0.5, 0.75)}] (0.5, 0.75) node[anchor = north west]{$a$} -- (0.5, - 0.25) node[anchor = south west]{$\nu_1$};
        \draw (1.4, 1.5) node[anchor = west]{$\Delta_0$};
    \end{tikzpicture}
    \hspace{2 cm}
    \begin{tikzpicture} [scale = 2]
        \fill[darkgreen!35] (- 1, - 1) -- (- 1, 1) -- (1, 1) -- (1, - 1) -- cycle;
        \fill[darkgreen] (- 1, - 0.5) -- (- 1, 0.3) -- (- 0.83, 0.55) -- (- 0.33, 0.3) -- (- 0.17, 0.55) -- (0.33, 0.3) -- (0.5, 0.55) -- (1, 0.3) -- (1, - 0.5) -- cycle;
        \draw (- 1, 0.3) -- (- 0.83, 0.55) -- (- 0.33, 0.3) -- (- 0.17, 0.55) -- (0.33, 0.3) -- (0.5, 0.55) -- (1, 0.3);
        \draw (- 1, - 0.5) -- (1, - 0.5);
        \draw[dashed] (- 1, 0.3) -- (1, 0.3);
        \draw[ultra thin, |-|] (- 0.33, 0.2) -- node[midway, below]{$\abs{z_0}/h$} (0.33, 0.2);
        \draw (- 1, - 1) -- (- 1, 1) -- (1, 1) -- (1, - 1) -- cycle;
        \draw[thick, ->](- 1, - 0.00001) -- (- 1, 0.00001);
        \draw[thick, ->](1, - 0.00001) -- (1, 0.00001);
        \draw[thick, ->](- 0.00001, - 1) -- (0.00001, - 1);
        \draw[thick, ->](- 0.00001, 1) -- (0.00001, 1);
        \draw(1.3, 0.8) node{$\T^2$};
        \draw(0.5, - 0.3) node{$E_h$};
    \end{tikzpicture}
    \caption{The triangle $\Delta_0$ described in \cref{rem:4.laminf}, and a minimizer $E_h$ of \eqref{eq:1.wulffper} obtained for $h = 3$.}
    \label{fig:4.laminf}
\end{figure}

\begin{rem} \label{rem:4.square}

    For any given anisotropy, there are at most four distinct preferred orientations for the optimal lamellae, 
    that is, the set $\D(K)$ contains at most four different pairs of opposite points (see \cite[Theorem 1.1]{DMN12}).
    This maximum of four directions is achieved by the $\ell^1$-unit ball $K = \conv \cb{\pm(1, 0),\pm(0, 1)}$.
    Indeed $\D(K) = \cb{\pm(1,0), \pm(0,1), \pm(1,1), \pm(1,-1)}$, that is, there are exactly four different preferred directions for the optimizing lamellae, 
    which are global minimizers only for $M = 1/2$ (as $\alpha(K) = 1/2$).
    Notice also that in this case there are infinitely many minimizers (\cref{thm:uniqueness}); the droplet $\frac{1}{2} K$ and the bubble $(- \frac{1}{2} K)^c$ coincide up to translation.

\end{rem}

\subsection{The case of a general lattice}

We conclude this section by discussing the case of a general lattice $\Lambda = A_\Lambda \Z^2$, for $A_\Lambda \in \GL(\R^2)$. 
We consider the corresponding torus $\T_\Lambda \defeq \R^2 / \Lambda$.
The linear transformation $A_\Lambda$ acts as a bijection between $\T^2$ and $\T_\Lambda$, so that, 
if $E \subset \T^2$ is a set of finite perimeter, then so is $A_\Lambda E \subset \T_\Lambda$, and for any $y \in \reb (A_\Lambda E) = A_\Lambda (\reb E)$ the unit normal is given by
$\nu_{A_\Lambda E}(y) = \frac{A_\Lambda^{- T} \nu_E(A_\Lambda^{- 1} y)}{\abs{ A_\Lambda^{- T} \nu_E(A_\Lambda^{- 1} y)}}$.

We then generalize problem \eqref{eq:1.wulffper} to the case of the torus $\T_\Lambda$:
\begin{equation} \label{eq:4.wulffper_latt}
    \min \cb{P_{\T_\Lambda}^\phi(E) : E \subseteq \T_\Lambda, \ \abs{E} = M}, \qquad \qquad M \in (0, d(\Lambda)),
\end{equation}
where $P_{\T_\Lambda}^\phi (E)$ is the anisotropic perimeter of $E$ on $\T_\Lambda$ corresponding to a surface tension $\phi$, with associated Wulff shape $K$.
By a change of variables \cite[Proposition~17.1]{Mag12},
\begin{equation} \label{eq:5.per_latt}
    P_{\T_\Lambda}^\phi (E) = d(\Lambda) \int_{\reb (A_\Lambda^{- 1} E)} \phi \rb{A_\Lambda^{- T} \nu_{A_\Lambda^{- 1} E}(x)} \de \Ha^1(x) 
    = P_{\T^2}^{\phi_\Lambda} (A_\Lambda^{- 1} E),
\end{equation}
where $\phi_\Lambda : \R^2 \to [0, + \infty)$ is the surface tension defined as $\phi_\Lambda(\nu) \defeq d(\Lambda) \phi(A_\Lambda^{- T} \nu)$, $\nu \in \R^2$.

We observe that lamellar configurations on $\T_\Lambda$ are images of lamellar configurations on $\T^2$ through the map $A_\Lambda$:
given $z \in \Z^2 \setminus \cb{0}$ with $\GCD(z_1, z_2) = 1$, the set $A_\Lambda L_M^{z}$ is a lamellar configuration on $\T_\Lambda$
whose normal direction is $A_\Lambda^{- T} z \in \Lambda^\circ \setminus \cb{0}$, and area equal to $d(\Lambda) M$.
As a consequence, the lamella $L_{\Lambda, M}^\nu \subset \T_\Lambda$ with area $M \in (0, d(\Lambda))$ and orthogonal direction $\nu \in \Lambda^\circ \setminus \cb{0}$ is defined as
\begin{equation*}
    L_{\Lambda, M}^\nu \defeq  A_\Lambda L_{M/d(\Lambda)}^{A_\Lambda^T \nu}.
\end{equation*}
Moreover, by \eqref{eq:5.per_latt} we have that
\begin{equation*}
    P_{\T_\Lambda}^\phi (L_{\Lambda, M}^\nu) = P_{\T^2}^{\phi_\Lambda} (A_\Lambda^{- 1} L_{\Lambda, M}^\nu) = P_{\T^2}^{\phi_\Lambda} \rb{L_{M/d(\Lambda)}^{A_\Lambda^T \nu}}
    = \phi_\Lambda \big( A_\Lambda^T \nu \big) + \phi_\Lambda \big( - A_\Lambda^T \nu \big) =  d(\Lambda) \big( \phi(\nu) + \phi(- \nu) \big).
\end{equation*}
As a consequence, 
\begin{equation*}
    \min_{\nu \in \Lambda^\circ \setminus \cb{0}} P_{\T_\Lambda}^\phi (L_{\Lambda, M}^\nu) = P_{\T_\Lambda}^\phi (L_{\Lambda, M}^{\bar \nu}) 
    = d(\Lambda) w_\Lambda(K) \quad \text{for all }\bar{\nu} \in \D_\Lambda(K).
\end{equation*}

\begin{rem} \label{rem:link_latt}

    Problem \eqref{eq:4.wulffper_latt} is equivalent to problem \eqref{eq:1.wulffper} in the standard lattice $\Z^2$ for the surface tension $\phi_\Lambda$: 
    $E$ is a minimizer of area $M \in (0, d(\Lambda))$ for \eqref{eq:4.wulffper_latt} if and only if $F \defeq A_\Lambda^{- 1} E$ solves
    \begin{equation} \label{eq:4.minproblem}
        \min \cb{P_{\T^2}^{\phi_\Lambda}(F) : F \subseteq \T^2, \ \abs{F} = \frac{M}{d(\Lambda)}}.
    \end{equation}
    The Wulff shape associated with the surface tension $\phi_\Lambda$ is given by $W_{\phi_\Lambda} = d(\Lambda) A_\Lambda^{- 1} K$,
    moreover $w(W_{\phi_\Lambda}) = d(\Lambda) w_{\Lambda}(K)$ and $\D_\Lambda(K) = A_\Lambda^{- T} \D(A_\Lambda^{- 1} K)$.
    
\end{rem}

Denoting by $\alpha_\Lambda(K) \defeq \frac{d(\Lambda)^2 w_\Lambda(K)^2}{4 \abs{K}}$, the previous remark allows us to deduce the following two extensions of \cref{thm:main} and \cref{thm:uniqueness} to the case of general lattices.

\begin{thm} \label{thm:4.main_latt}

    Let $\phi : \R^2 \to [0, + \infty)$ be a surface tension with $K \subset \R^2$ the corresponding Wulff shape, and let $\Lambda \subset \R^2$ be a lattice.
    Then, depending on the area constraint $M \in (0, d(\Lambda))$, the following sets solve the periodic Wulff problem \eqref{eq:4.wulffper_latt} on $\T_\Lambda$:
    \begin{equation*}
        \begin{dcases}
            \sqrt{\tfrac{M}{\abs{K}}} K \qquad & \text{ if } \quad 0 < M \le \alpha_\Lambda(K) \wedge \tfrac{d(\Lambda)}{2}, \\[1ex]
            L_{\Lambda, M}^\nu \quad \text{ for } \nu \in \D_\Lambda(K) \qquad & \text{ if } \quad \alpha_\Lambda(K) \le M \le d(\Lambda) - \alpha_\Lambda(K), \\[1ex]
            \rb{- \sqrt{\tfrac{d(\Lambda) - M}{\abs{K}}} K}^c \qquad & \text{ if } \quad d(\Lambda) - \Bigl(\alpha_\Lambda(K) \wedge \tfrac{d(\Lambda)}{2}\Bigr) \le M < d(\Lambda).
        \end{dcases}
    \end{equation*}
    If $\phi$ is symmetric, then $\alpha_\Lambda(K) \le \frac{d(\Lambda)}{2}$ and the lamella is always a minimizer for some value of $M$.
    
\end{thm}

\begin{proof}

    The assertion follows combining \cref{thm:main} and \cref{rem:link_latt}, with \cref{cor:2.symarea_latt} yielding the estimate on $\alpha_\Lambda(K)$ in the symmetric case.
    \qedhere
    
\end{proof}

\begin{thm} \label{thm:4.uniqueness_latt}

    Let $\phi : \R^2 \to [0, + \infty)$ be a surface tension with $K \subset \R^2$ the corresponding Wulff shape, and let $\Lambda \subset \R^2$ be a lattice.
    Then, problem \eqref{eq:4.wulffper_latt} admits minimizers which are not equal, up to translations and Lebesgue-negligible sets, to any of the minimizers listed in \cref{thm:4.main_latt} 
    if and only if both of the following conditions hold:
    \begin{enumerate}
        \item 
            $\alpha_\Lambda(K) \le d(\Lambda)/2$ and $M \in [\alpha_\Lambda(K), d(\Lambda) - \alpha_\Lambda(K)]$;
            
        \item 
            there exists $z_0 \in \D_\Lambda(K)$ such that $\phi$ is linear in a neighborhood of $z_0$.
            
    \end{enumerate}
    In this case there are infinitely many minimizers for the same value of $M$.
    
\end{thm}

\begin{proof}

    This is proved by combining \cref{thm:uniqueness} and \cref{rem:link_latt}.
    \qedhere
    
\end{proof}

\section{On the uniqueness of minimizers} \label{sec:uniq}

In this section we give the proof of \cref{thm:uniqueness}. 
For a surface tension satisfying the conditions in the statement, the existence of additional minimizers, 
other than those listed in \cref{thm:main}, has been proved in \cref{rem:4.laminf}. 
Hence we focus here only on the necessity of the two conditions.

We premise to the proof of \cref{thm:uniqueness} a couple of topological lemmas.
We start by recalling the following result on liftings, which is a special case of the lifting criterion (see for instance \cite[Theorem~5.5.1]{Mas77}).

\begin{lem} \label{lem:5.lifting}

    Let $X$ be a simply connected and locally path-connected topological space.
    Then, for any continuous map $f: X \to \T^2$ there exists a continuous map $\tilde f: X \to \R^2$ such that $\pi \circ \tilde f = f$, 
    where $\pi : \R^2 \to \T^2$ denotes the natural projection.
    
\end{lem}

For the next lemma, we recall some basic facts in algebraic topology, see for instance \cite[Chapter~1]{Mas77}.
For integers $g, b \ge 0$, let $S_{g, b}$ be the surface obtained by removing $b$ open disks with pairwise disjoint closures 
from the sphere (if $g = 0$) or from the connected sum of $g$ tori (if $g \ge 1$). 
Let $\Sigma \subset \T^2$ be a closed, connected set with smooth boundary and nonempty interior. 
Then $\Sigma$ is a compact, connected surface with boundary, which is orientable since it inherits the orientation of $\T^2$. 
By the classification of compact surfaces with boundary, $\Sigma$ is homeomorphic to $S_{g, b}$ for unique integers $g \ge 0$ (the \emph{genus} of $\Sigma$) 
and $b \ge 0$ (the number of connected components of $\partial\Sigma$), and the \emph{Euler characteristic} of $\Sigma$ can be computed according to the formula
\begin{equation} \label{eq:5.euler}
    \chi(\Sigma) = 2 - 2 g - b.
\end{equation}
We then have the following lemma.

\begin{lem} \label{lem:5.euler}

    Let $\Sigma \subset \T^2$ be a proper closed set with smooth boundary and nonempty interior.
    Assume that both $\Sigma$ and $\Sigma^c$ are connected.
    Then exactly one of the following holds:
    \begin{enumerate}
        \item \label{item:5.disk}
            $\chi(\Sigma) = 1$, and $\Sigma$ is homeomorphic to a closed disk;
            
        \item \label{item:5.diskc}
            $\chi(\Sigma) = - 1$, and $\overline{\Sigma^c}$ is homeomorphic to a closed disk;

        \item \label{item:5.annulus}
            $\chi(\Sigma) = 0$, and both $\Sigma$ and $\overline{\Sigma^c}$ are homeomorphic to the closed annulus $\S^1\times[0,1]$. 
            In this case $\partial\Sigma$ consists of two disjoint, smooth, closed, simple curves, which are non-contractible in $\T^2$.
            
    \end{enumerate}
    
\end{lem}

\begin{proof}

    Set $\Sigma' \defeq \overline{\Sigma^c}$. Since $\Sigma$ and $\Sigma'$ satisfy $\T^2 = \Sigma \cup \Sigma'$ and $\Sigma \cap \Sigma' = \partial \Sigma$, we have
        \begin{equation*}
            0 = \chi(\T^2) = \chi(\Sigma) + \chi(\Sigma') - \chi(\Sigma \cap \Sigma') = \chi(\Sigma) + \chi(\Sigma'),
        \end{equation*}
    where $\chi(\partial\Sigma) = 0$ since $\partial\Sigma$ consists of $b \ge 1$ disjoint, smooth, simple, closed curves 
    (notice that $\partial \Sigma \ne \emptyset$ as $\Sigma \ne \emptyset, \T^2$ and $\T^2$ is connected).
    
    Since $b \ge 1$, \eqref{eq:5.euler} yields $\chi(\Sigma) \le 1$, $\chi(\Sigma') \le 1$. 
    Combined with $\chi(\Sigma) + \chi(\Sigma') = 0$, we obtain $\chi(\Sigma) \in \cb{1, 0, - 1}$, and correspondingly we get three mutually exclusive cases:
    \begin{itemize}
        \item 
            if $\chi(\Sigma) = 1$, then $2 g + b = 1$, hence $g = 0$, $b = 1$, and $\Sigma$ is homeomorphic to $S_{0, 1}$, that is, to a closed disk;

        \item 
            if $\chi(\Sigma) = - 1$, then $\chi(\Sigma') = 1$, and likewise we get that $\Sigma'$ is homeomorphic to a closed disk;

        \item 
            if $\chi(\Sigma) = \chi(\Sigma') = 0$, then $2 g + b = 2 = 2 g' + b$ (where $g'$ is the genus of $\Sigma'$) with $b \ge 1$, 
            which forces $g = g' = 0$, $b = 2$, and both $\Sigma$ and $\Sigma'$ are homeomorphic to $S_{0, 2}$, that is, to a closed annulus. 
            In particular $\partial \Sigma = \Gamma_1 \cup \Gamma_2$ for two disjoint, smooth, simple, closed curves.
            
    \end{itemize}
    Finally, in the third case we need to establish the non-contractibility in $\T^2$ of the two boundary curves. 
    Observe that $\T^2 \setminus \Gamma_1 = (\Sigma \setminus \Gamma_1) \cup (\Sigma' \setminus \Gamma_1)$ is connected, 
    as both $\Sigma \setminus \Gamma_1$ and $\Sigma' \setminus \Gamma_1$ are connected and both contain $\Gamma_2$. 
    If $\Gamma_1$ were contractible in $\T^2$, then $\Gamma_1$ would bound a closed disk in $\T^2$ (see \cite[Theorem~1.7]{Eps66}) 
    and $\T^2 \setminus \Gamma_1$ would be disconnected, which is a contradiction.
    Similarly, one proves that $\Gamma_2$ is not contractible.
    \qedhere 
    
\end{proof}

\begin{proof}[Proof of \cref{thm:uniqueness} (necessity)]

    Let $E \subset \T^2$ be a minimizer for \eqref{eq:1.wulffper} with $\abs{E} = M \in (0, 1)$,
    and assume that $E$ is not equal, up to translations and Lebesgue-negligible sets, to any of the minimizers listed in \cref{thm:main}.
    We divide the proof into five steps.
    
    \vspace{5 pt}
    \noindent
    \textit{Step one (indecomposability of $E$ and $E^c$)}:
    we claim that $E$ is indecomposable, that is, if two measurable sets $A, B \subset \T^2$ satisfy 
    $\abs{(A \cup B) \Delta E} = 0$, $\abs{A \cap B} = 0$, and $P_{\T^2}(E) = P_{\T^2}(A) + P_{\T^2}(B)$, then  $\abs{A} \cdot \abs{B} = 0$.
    
    Indeed, the isoperimetric profile $I_{\T^2}^\phi (M) = 2 \abs{K}^{1/2} \min\{ M, \alpha(K), 1-M \}^{1/2}$, see \eqref{eq:1.wulffvalues}, is a strictly subadditive function:
    \begin{equation} \label{eq:3.profilemin}
        I_{\T^2}^\phi(M_1 + M_2) < I_{\T^2}^\phi(M_1) + I_{\T^2}^\phi(M_2) \qquad \qquad \text{ for any } M_1, M_2 > 0 \text{ with } M_1 + M_2 < 1.
    \end{equation}
    Moreover, notice that the equality $P_{\T^2}(A \cup B) = P_{\T^2}(A) + P_{\T^2}(B)$ 
    implies $\Ha^1(\reb A \cap \reb B) = 0$,
    and as a consequence $P_{\T^2}^\phi (A \cup B) = P_{\T^2}^\phi (A) + P_{\T^2}^\phi (B)$.
    Thus
    \begin{equation*}
        I_{\T^2}^\phi (M) = P_{\T^2}^\phi(A \cup B) = P_{\T^2}^\phi(A) + P_{\T^2}^\phi(B) \ge I_{\T^2}^\phi (\abs{A}) + I_{\T^2}^\phi (\abs{B}).
    \end{equation*}
    Since $\abs{A}+\abs{B}=M$, by \eqref{eq:3.profilemin} this is possible only if $\abs{A} \cdot \abs{B} = 0$, proving the claim.
    Similarly, since $E^c$ is a minimizer of \eqref{eq:1.wulffper} with area $\abs{E^c} = 1 - M$ for the surface tension $\tilde \phi(x) = \phi(- x)$, 
    we see that $E^c$ is indecomposable as well.

    \vspace{5 pt}
    \noindent
    \textit{Step two (approximation of $E$)}:
    let us prove the existence of a sequence $\cb{E_h}_{h \in \N}$ of closed sets of finite perimeter with smooth boundary such that
    \begin{itemize}
        \item 
            $E_h \to E$ in $L^1$ (i.e., $\norm{\1_{E_h} - \1_E}_{L^1(\T^2)} \to 0$ as $h \to \infty$);

        \item 
            $\abs{D \1_{E_h}} \starto \abs{D \1_E}$ as $h \to \infty$;

        \item 
            $E_h$ and $E_h^c$ are connected for any $h \in \N$.
            
    \end{itemize}
    By a standard argument we may find a sequence $\{ \tilde E_h \}_{h \in \N}$ of closed sets of finite perimeter with smooth boundary such that
    the first two properties hold (see for instance \cite[Theorem 13.8]{Mag12}).
    However, in general such a sequence does not satisfy the third property.
    Let $\{ \tilde E_h^i \}_{i = 1}^{N_h}$ be the finitely many connected components of $\tilde E_h$, where, up to reordering, 
    we may assume that $|\tilde E_h^1| \ge |\tilde E_h^i|$ for $1 < i \le N_h$,
    and by smoothness of the boundary we also have $\partial \tilde E_h^i \cap \partial \tilde E_h^j = \emptyset$ for $i \ne j$.
    Notice that $\liminf_{h \to \infty} |\tilde E_h^1| > 0$: indeed, on noting that $|\tilde E_h^i| \le \frac{1}{2}$ for $i = 2, \dots, N_h$, 
    by the isoperimetric inequality \eqref{eq:1.isop} in $\T^2$ we have
    \begin{equation*}
        |\tilde E_h| - |\tilde E_h^1| = \sum_{i = 2}^{N_h} |\tilde E_h^i| \le C \sum_{i = 2}^{N_h} |\tilde E_h^i|^{1/2} P_{\T^2}(\tilde E_h^i)
        \le C |\tilde E_h^1|^{1/2} \sum_{i = 2}^{N_h} P_{\T^2}(\tilde E_h^i) \le C |\tilde E_h^1|^{1/2} P_{\T^2}(\tilde E_h),
    \end{equation*}
    which implies that $|\tilde E_h^1|$ cannot go to zero since $|\tilde E_h|\to M \in (0,1)$ and $P_{\T^2}(\tilde{E}_h) \to P_{\T^2}(E)$.

    Since the sets $\tilde E_h^1$ have uniformly bounded perimeter, up to subsequences $\tilde E_h^1 \to \tilde E$ in $L^1$ 
    for some set of finite perimeter $\tilde E$ with $|\tilde E \setminus E| = 0$, and in particular $|\tilde E| = \lim_{h \to \infty} |\tilde E_h^1| > 0$.
    Moreover, $\tilde E_h \setminus \tilde E_h^1 \to E \setminus \tilde E$ in $L^1$ and by lower semicontinuity and subadditivity of the perimeter we obtain
    \begin{align*}
        P_{\T^2}(E) &= \lim_{h \to \infty} P_{\T^2}(\tilde E_h) = \lim_{h \to \infty}  P_{\T^2} \rb{\tilde E_h^1 \cup (\tilde E_h \setminus \tilde E_h^1)}
        = \lim_{h \to \infty} \rb{P_{\T^2}(\tilde E_h^1) + P_{\T^2}(\tilde E_h \setminus \tilde E_h^1)} \\
        &\ge \liminf_{h \to \infty} P_{\T^2}(\tilde E_h^1) + \liminf_{h \to \infty} P_{\T^2} (\tilde E_h \setminus \tilde E_h^1)
        \ge P_{\T^2}(\tilde E) + P_{\T^2}(E \setminus \tilde E) \ge P_{\T^2}(E).
    \end{align*}
    Therefore all inequalities in the previous chain are equalities, and since $|\tilde E| > 0$, by indecomposability of $E$ we conclude that $|E \setminus \tilde E| = 0$.
    Consequently $\tilde E_h^1 \to E$, $P_{\T^2}(\tilde E_h^1) \to P_{\T^2}(E)$, hence $|D \1_{\tilde E_h^1}| \starto \abs{D \1_E}$, 
    and we may replace $\{ \tilde E_h \}_{h \in \N}$ by the sequence of connected sets $\{ \tilde E_h^1 \}_{h \in \N}$.

    We finally need to establish that also the complements are connected. We repeat the same procedure on the complements $C_h \defeq (\tilde E_h^1)^c$: 
    that is, we order the connected components $\{ C_h^i \}_{i = 1}^{K_h}$ of $C_h$ so that $|C_h^1| \ge |C_h^i|$ for $1 < i \le K_h$, and we define $E_h \defeq \T^2 \setminus C_h^1$. 
    Arguing as before with the sequence $\overline{C_h}$ converging in $L^1$ to the indecomposable set $E^c$, 
    we find that, up to subsequences, $E_h \to E$ and $\abs{D \1_{E_h}} \starto \abs{D \1_E}$. 
    Moreover, $E_h^c$ is clearly connected; $E_h = \tilde E_h^1 \cup \bigcup_{i = 2}^{K_h} \overline{C_h^i}$ is also connected, 
    since each $\overline{C_h^i}$ is connected and meets $\tilde E_h^1$.
    
    \vspace{5 pt}
    \noindent
    \textit{Step three ($\chi(E_h) = 0$ for all $h$ sufficiently large)}:
    we can apply \cref{lem:5.euler} to the sets $E_h$, noticing that $|E_h|\to M\in(0,1)$.
    Assume by contradiction that $\chi(E_h) = 1$ for infinitely many $h \in \N$; up to subsequences, every $E_h$ is homeomorphic to a closed disk. 
    By \cref{lem:5.lifting} applied to the inclusion map $i : E_h \to \T^2$, we obtain a continuous map $\tilde \imath: E_h \to \R^2$ such that $\pi \circ \tilde \imath = i$. 
    Setting $F_h \defeq \tilde \imath (E_h)$, we have that $F_h$ is a compact subset of $\R^2$ homeomorphic to $E_h$ 
    (the homeomorphism is given by the restriction of $\pi$ to $F_h$, whose inverse is the continuous map $\tilde\imath$).
    
    It is easily verified that the periodic extension $\pi^{- 1}(E_h)$ of $E_h$ to $\R^2$ is the disjoint union of the translated copies of $F_h$:
    \begin{equation*}
        \pi^{- 1}(E_h) = \bigcup_{z \in \Z^2} (F_h + z),
    \end{equation*}
    with $F_h \cap (F_h + z) = \emptyset$ for all $z \in \Z^2 \setminus \cb{0}$ since otherwise two distinct points of $F_h$ would have the same image under $\pi$. 
    By compactness of $F_h$, we also deduce that all the sets $F_h + z$, for $z \in \Z^2$, have positive mutual distance; 
    since $\pi$ is a local diffeomorphism, it follows that $F_h$ has smooth boundary $\partial F_h = F_h \cap \pi^{- 1}(\partial E_h)$, 
    which $\pi$ maps homeomorphically onto the closed curve $\partial E_h$. 
    As a consequence, setting $Q \defeq [0, 1)^2$, we have that
    \begin{equation*}
        \abs{E_h} = \big| \pi^{- 1} (E_h) \cap Q \big| = \sum_{z \in \Z^2} \big| F_h \cap (Q - z) \big| = \abs{F_h},
    \end{equation*}
    and, by \eqref{eq:2.relper},
    \begin{equation*}
        P^\phi_{\T^2} (E_h) = P^\phi(\pi^{- 1}(E_h); Q) = \sum_{z \in \Z^2} P^{\phi}(F_h; Q - z) = P^\phi(F_h).
    \end{equation*}
    The sets $F_h$ have uniformly bounded perimeter and in turn, being connected, uniformly bounded diameters. 
    Since translating $F_h$ by a vector of $\Z^2$ does not affect the previous properties, 
    we may assume that the sets $F_h$ are contained in a fixed ball, and by the compactness theorem for sets of finite perimeter, 
    up to a subsequence $F_h \to F$ in $L^1(\R^2)$ for some bounded set of finite perimeter $F \subset \R^2$ with $\abs{F} = M$, and $\pi(F) = E$.

    Since $\abs{D \1_{E_h}} \starto \abs{D \1_E}$, by Reshetnyak continuity theorem \cite[Theorem~2.39]{AFP00} we have that
    \begin{equation*}
         P_{\T^2}^\phi(E) = \lim_{h \to \infty} P_{\T^2}^\phi(E_h) = \lim_{h \to \infty} P^\phi(F_h) \ge P^\phi(F)
         \ge 2 \sqrt{M \abs{K}} \ge I_{\T^2}^\phi(M) = P_{\T^2}^\phi(E),
    \end{equation*}
    where we used the lower semicontinuity of the anisotropic perimeter in the third passage, the Wulff inequality \eqref{eq:2.per-wulff} in $\R^2$ in the fourth, 
    the explicit expression of the isoperimetric profile \eqref{eq:1.wulffvalues} in the fifth, and the minimality of $E$ in the last equality. 
    Therefore all inequalities in the previous chain are equalities.
    In particular, $0 < M \le \alpha(K) \wedge \frac{1}{2}$, and by rigidity of the Wulff inequality $F = \sqrt{\frac{M}{\abs{K}}} K$ up to translations and null sets. 
    This is a contradiction, since then $E =\pi(F)$ would be one of the sets listed in \cref{thm:main}.

    Similarly, if $\chi(E_h) = - 1$ for infinitely many $h \in \N$, by repeating the same argument for $E^c$ (which is a minimizer for the anisotropy $\tilde \phi(x)=\phi(-x)$) 
    and the approximating sequence $\{ \overline{E_h^c} \}_{h \in \N}$ we would again obtain a contradiction.

    \vspace{5 pt}
    \noindent
    \textit{Step four (structure of $\reb E$)}:
    by the previous step and \cref{lem:5.euler}, we may assume without loss of generality that $\chi(E_h) = 0$ for any $h \in \N$: 
    $E_h$ and $\overline{E_h^c}$ are homeomorphic to a closed annulus, and $\partial E_h = \Gamma_{1, h} \cup \Gamma_{2, h}$ is the disjoint union of two loops, 
    parameterized respectively by two smooth, non-contractible curves $\gamma_{i, h} : [0, 1] \to \T^2$ injective in $[0, 1)$, 
    with $\gamma_{i, h}(0) = \gamma_{i, h}(1)$ and $|\gamma_{i, h}'| \equiv \Ha^1(\Gamma_{i, h})$, $i = 1, 2$. 
    We orient the two parametrizations so that
    \begin{equation*}
        \nu_{E_h}(\gamma_{i, h}(t)) = \frac{(\gamma_{i, h}'(t))^\perp}{\Ha^1(\Gamma_{i, h})} \qquad \text{for all } t \in [0, 1], \ i = 1, 2,
    \end{equation*}
    where $(a, b)^\perp \defeq (-b, a)$.
    
    Since $\Ha^1(\Gamma_{1, h}) + \Ha^1(\Gamma_{2, h}) = P_{\T^2}(E_h) \to P_{\T^2}(E)$, the curves $\gamma_{i, h}$ are equi-Lipschitz and, by the Ascoli--Arzelà theorem, 
    up to a subsequence $\gamma_{i, h} \to \gamma_i$ uniformly in $[0, 1]$, $i = 1, 2$, for two closed, Lipschitz curves $\gamma_i : [0, 1] \to \T^2$. 
    We set $\Gamma_i \defeq \gamma_i([0, 1])$.

    Notice that $\reb E \subset \spt\abs{D \1_E} \subset \Gamma_1\cup\Gamma_2$: indeed, if $x \not \in \Gamma_1 \cup \Gamma_2$, 
    choose $r > 0$ such that $B_r(x) \cap (\Gamma_1 \cup \Gamma_2) = \emptyset$. 
    Since $\gamma_{i, h} \to \gamma_i$ uniformly, $B_{r/2}(x) \cap \partial E_h = \emptyset$ for $h$ large. 
    Then $\abs{D \1_E}(B_{r/2}(x)) \le \liminf_{h \to \infty} \abs{D \1_{E_h}}(B_{r/2}(x)) = 0$, hence $x \not \in \spt \abs{D \1_E}$.
    
    It follows that $P_{\T^2}(E) \le \Ha^1(\Gamma_1\cup\Gamma_2)$. By weak$^\star$-convergence of $\gamma_{i,h}' \starto \gamma_i'$ in $L^\infty((0,1);\R^2)$, we get
    \begin{align*}
        P_{\T^2}(E)
        & \le \Ha^1(\Gamma_1 \cup \Gamma_2) \le \Ha^1(\Gamma_1) + \Ha^1(\Gamma_2) \le \int_0^1 \abs{\gamma_1'(t)} \de t + \int_0^1 \abs{\gamma_2'(t)} \de t \\
        & \le \liminf_{h \to \infty} \rb{\int_0^1 \abs{\gamma_{1, h}'(t)} \de t + \int_0^1 \abs{\gamma_{2, h}'(t)} \de t} = \lim_{h \to \infty} P_{\T^2}(E_h) = P_{\T^2}(E).
    \end{align*}
    We deduce that all inequalities in the previous chain are equalities: in particular, it follows that $\Ha^1(\Gamma_1 \cap \Gamma_2) = 0$, 
    $\Ha^1 \rb{\reb E \Delta (\Gamma_1 \cup \Gamma_2)} = 0$, and $\Ha^1(\Gamma_i) = \int_0^1 \abs{\gamma_i'(t)} \de t$. 
    The last identity, combined with the area formula \cite[Theorems~8.7 and 8.9]{Mag12}, implies that for $\Ha^1$-a.e. $x \in \Gamma_i$ 
    there is a unique $t \in [0, 1)$ such that $\gamma_i(t) = x$, $\gamma_i'(t)$ exists and does not vanish; we may then set
    \begin{equation*}
        \tau_i(x) \defeq \frac{\gamma_i'(t)}{|\gamma_i'(t)|}.
    \end{equation*}
    By the divergence theorem and the choice of orientation, for all $T \in C^1(\T^2; \R^2)$
    \begin{equation*}
        \int_{E_h} \dive T = \sum_{i = 1}^2 \int_0^1 T(\gamma_{i, h}(t)) \cdot \gamma_{i, h}'(t)^\perp \de t,
    \end{equation*}
    and by letting $h\to\infty$, using the divergence theorem on $E$ and the area formula on $\Gamma_i$,
    \begin{equation*} 
        \int_{\reb E} T \cdot \nu_E \de \Ha^1 = \int_E \dive T \de x
        = \sum_{i = 1}^2 \int_0^1 T(\gamma_i(t)) \cdot \gamma_i'(t)^\perp \de t
        = \sum_{i = 1}^2 \int_{\Gamma_i} T \cdot \tau_i^\perp \de\Ha^1.
    \end{equation*}
    Since $\Ha^1(\Gamma_1\cap\Gamma_2)=0$ and $\Ha^1 \rb{\reb E \Delta (\Gamma_1 \cup \Gamma_2)} = 0$, by the arbitrariness of $T$ we infer that
    \begin{equation*}
        \nu_E(x) = \tau_i(x)^\perp \qquad \text{for $\Ha^1$-a.e. } x \in \Gamma_i, \ i = 1, 2.
    \end{equation*}
    In turn we deduce, by the area formula, that
    \begin{equation} \label{eq:5.proof-step4}
        \int_{\Gamma_i} \nu_E \de \Ha^1 = \bigg( \int_0^1 \gamma_i'(t) \de t \bigg)^\perp.
    \end{equation}

    \vspace{5 pt}
    \noindent
    \textit{Step five (lamella comparison)}:
    for a closed Lipschitz curve $\gamma : [0, 1] \to \T^2$, we denote by $\tilde \gamma : [0, 1] \to \R^2$ a lifting of $\gamma$, given by \cref{lem:5.lifting}. 
    We define $w \defeq \int_0^1 \gamma'(t) \de t = \tilde \gamma(1) - \tilde \gamma(0)$. 
    We have $w \in \Z^2$ since $\pi(\tilde \gamma(1)) = \gamma(1) = \gamma(0) = \pi(\tilde \gamma(0))$. 
    Moreover, if $w = 0$ then $\gamma$ is contractible in $\T^2$, as in this case $\tilde \gamma$ is a closed loop in $\R^2$, hence contractible, and so is $\gamma = \pi \circ \tilde \gamma$.

    We apply this to the curves $\gamma_i$ and $\gamma_{i,h}$, setting
    \begin{equation*}
        w_i \defeq \int_0^1 \gamma_i'(t) \de t, \qquad w_{i, h} \defeq \int_0^1 \gamma_{i, h}'(t)\de t.
    \end{equation*}
    Since the curves $\gamma_{i, h}$ are non-contractible, we have $w_{i, h} \in \Z^2 \setminus \cb{0}$ for every $h \in \N$. 
    Moreover, $w_{i, h}\to w_i$ as $h \to \infty$, by the weak$^\star$ convergence of $\gamma_{i, h}' \starto \gamma_i'$ in $L^\infty((0, 1); \R^2)$. 
    As a convergent sequence of integer vectors is eventually constant, $w_i = w_{i, h}$ for $h$ large; in particular, $w_i \ne 0$, $i = 1,2$. 
    By the divergence theorem and Step four,
    \begin{equation*}
        0 = \int_{\reb E} \nu_E \de \Ha^1 = \int_{\Gamma_1} \nu_E \de \Ha^1 + \int_{\Gamma_2} \nu_E \de \Ha^1 \xupref{eq:5.proof-step4}{=} w_1^\perp + w_2^\perp,
    \end{equation*}
    that is, $w_1 = - w_2$. We set $z_0\defeq w_1^\perp = - w_2^\perp$.
    
    By definition of $w(K)$, since $z_0 \in \Z^2 \setminus \cb{0}$, using Jensen's inequality and positive one-homogeneity of $\phi$, we find
    \begin{align*}
        w(K) & \leq \phi(z_0) + \phi(- z_0) = \phi \rb{\int_{\Gamma_1} \nu_E \de \Ha^1} + \phi \rb{\int_{\Gamma_2} \nu_E \de \Ha^1} \\
        & \le \int_{\Gamma_1} \phi \rb{\nu_E} \de \Ha^1 + \int_{\Gamma_2} \phi \rb{\nu_E} \de \Ha^1
        = \int_{\reb E} \phi(\nu_E) \de \Ha^1 = P_{\T^2}^\phi(E) = I_{\T^2}^\phi(M) \le w(K)
    \end{align*}
    by minimality of $E$. Hence all inequalities are equalities: 
    by \eqref{eq:1.wulffvalues} we deduce that $\alpha(K) \le 1/2$ and $M \in [\alpha(K), 1 - \alpha(K)]$, and by definition of $\D(K)$ that $z_0 \in \D(K)$. 
    Furthermore,
    \begin{equation}\label{eq:5.proof-step5}
        \phi \rb{\int_{\Gamma_i} \nu_E \de \Ha^1} = \int_{\Gamma_i} \phi \rb{\nu_E} \de \Ha^1, \qquad i = 1, 2.
    \end{equation}

    Let $x_i \in \partial K$ be such that $\phi(w_i^\perp) = x_i \cdot w_i^\perp$, $i = 1, 2$. 
    The set $A_i \defeq \cb{\nu \in \R^2: \phi(\nu) = x_i \cdot \nu}$ is a closed convex cone containing $w_i^\perp$, 
    and it contains no line since $\phi(\nu) + \phi(- \nu) > 0$ for $\nu \ne 0$. 
    If $z_0 = w_1^\perp \in \Int A_1$, then $\phi$ is linear in a neighborhood of $z_0$, and the proof is concluded. 
    The same holds if $- z_0 = w_2^\perp \in \Int A_2$.
    
    Otherwise, assume that $w_i^\perp \in \partial A_i$, $i = 1, 2$. 
    We show that in this case $\nu_E$ is constant on each $\Gamma_i$ and in turn $E$ is a lamella, contradicting the initial assumption. 
    By \eqref{eq:5.proof-step4} and \eqref{eq:5.proof-step5}
    \begin{equation*}
        \int_{\Gamma_i} \Big( \phi(\nu_E) - x_i \cdot \nu_E \Big) \de \Ha^1
        = \int_{\Gamma_i} \phi(\nu_E) \de \Ha^1 - x_i \cdot w_i^\perp
        = \int_{\Gamma_i} \phi(\nu_E) \de \Ha^1 - \phi(w_i^\perp) = 0,
    \end{equation*}
    for $i = 1, 2$. Since $\phi(\nu) \ge x_i \cdot \nu$ for all $\nu \in \R^2$, the integrand on the left-hand side is nonnegative, thus
    \begin{equation*}
        \phi(\nu_E) = x_i \cdot \nu_E \qquad \Ha^1\text{-a.e. on } \Gamma_i,
    \end{equation*}
    that is, $\nu_E \in A_i$ $\Ha^1$-a.e. on $\Gamma_i$.
    Since $w_i^\perp \in \partial A_i$, there exists $\ell_i \in \S^1$ such that $\ell_i \cdot w_i^\perp = 0$, $\ell_i \cdot \nu \ge 0$ for all $\nu \in A_i$. 
    Then $\int_{\Gamma_i} \ell_i \cdot \nu_E \de \Ha^1 = \ell_i \cdot w_i^\perp =0 $, hence $\ell_i \cdot \nu_E = 0$ $\Ha^1$-a.e. on $\Gamma_i$, that is, 
    $\nu_E \in A_i \cap \R w_i^\perp= \cb{t w_i^\perp : t \ge 0}$, as $A_i$ contains no line. 
    Therefore
     \begin{equation*}
        \nu_E = \frac{w_i^\perp}{|w_i^\perp|} \qquad\text{$\Ha^1$-a.e.\ on $\Gamma_i$, $i=1,2$.}
     \end{equation*}
    Since $\nu_E = \tau_i^\perp$ by Step four, we get $\tau_i = \frac{w_i}{\abs{w_i}}$, that is, $\gamma_i'$ is a nonnegative multiple of $w_i$ a.e., 
    and $\Gamma_i = \pi(\tilde{\gamma}_i(0)+\R w_i)$ is a closed line. 
    As $w_1 = -w_2$ and $\Ha^1(\Gamma_1 \cap \Gamma_2)=0$, the two lines $\Gamma_1$ and $\Gamma_2$ are parallel and disjoint. 
    Finally, since $\spt \abs{D\1_E} \subset \Gamma_1 \cup \Gamma_2$, $E$ coincides, up to a null set, with one of the two strips bounded by $\Gamma_1$ and $\Gamma_2$, 
    that is, with a translate of $L_M^{z_0}$. This contradicts the initial assumption and concludes the proof of the theorem.
    \qedhere
    
\end{proof}

\section{An application to the anisotropic Ohta--Kawasaki functional} \label{sec:6}

As an application of \cref{thm:main}, we briefly discuss the global minimality of lamellar configurations for an anisotropic version of the Ohta--Kawasaki energy.
This extends a recent result by the second author \cite{Fio26}, whereas the isotropic case was previously obtained by Sternberg and Topaloglu \cite{ST11} (see also \cite{AFM13}).
For $E \subset \T^2$ we let
\begin{equation} \label{eq:5.anOK}
    J_\gamma^\phi(E) = P_{\T^2}^\phi(E) + \gamma \int_{\T^2} \rb{\int_{\T^2} G(x, y) u_E(x) u_E(y) \de x} \de y,
\end{equation}
where $\gamma \ge 0$ is a fixed parameter, $u_E \defeq \1_E - \1_{E^c}$, and $G$ is the Green's function on the torus for $- \Delta$,
and we consider the corresponding minimum problem
\begin{equation} \label{eq:5.anOKprob}
    \inf \Big\{ J_\gamma^\phi(E) : E \subset \T^2, \ \abs{E} = M \Big\}, \qquad M \in (0, 1).
\end{equation}
It is well known that the presence of the nonlocal term in the energy \eqref{eq:5.anOK} affects the geometry of optimal configurations, as large values of $\gamma$ favor oscillations: 
minimizers are expected to be periodic on an intermediate length scale determined by $\gamma$. 

The result presented in this section holds for surface tensions which are either uniformly elliptic (see \cref{defi:2.uniell})
or \emph{horizontally flat} according to the following definition.

\begin{defi}

    A surface tension $\phi : \R^2 \to [0, + \infty)$ with Wulff shape $K$ is said to be \textit{horizontally flat} 
    if the subdifferential $\partial \phi$ of $\phi$ at $\pm e_2$ satisfies
    \begin{equation*}
        \Ha^1 \big( \partial \phi (\pm e_2) \big) = \Ha^1 \Big( \big\{ x \in \partial K : \pm x_2 = \phi(\pm e_2) \big\} \Big) > 0.
    \end{equation*}
    
\end{defi}

This simply amounts to requiring that the Wulff shape $K$ of $\phi$ displays upper and lower horizontal facets on its boundary.
In particular, notice that $\phi$ cannot be linear in any neighborhood of $\pm e_2$, otherwise $\phi$ would be differentiable at $\pm e_2$, against the assumption of horizontal flatness.

The following criterion, based on ideas first introduced in \cite{AFM13} for the isotropic case,
establishes that, whenever the horizontal lamella is the unique minimizer of the periodic Wulff problem, then it also solves \eqref{eq:5.anOKprob} for small values of $\gamma$.

\begin{prop}[{\cite[Proposition~7.1]{Fio26}}] \label{prop:5.stillmin}

    Let the surface tension $\phi$ be either uniformly elliptic or horizontally flat.
    If the horizontal lamella $L_M^{e_2}$ of area $M \in (0, 1)$ 
    is the unique (up to translations) global minimizer of the periodic Wulff problem \eqref{eq:1.wulffper}, 
    then it is also the unique global minimizer of $J_\gamma^\phi$, provided $\gamma > 0$ is sufficiently small.
    
\end{prop}

Combining this result with \cref{thm:main} and \cref{thm:uniqueness} we get the next corollary, which extends \cite[Theorem~2.13]{Fio26}.

\begin{cor}

    Let $\phi : \R^2 \to [0, + \infty)$ be uniformly elliptic or horizontally flat and let $K \subset \R^2$ be the corresponding Wulff shape.
    Assume also that $\D(K) = \cb{\pm e_2}$ and $\alpha(K) < \frac{1}{2}$.
    Then, for $M \in (\alpha(K), 1 - \alpha(K))$ the horizontal lamella $L_M^{e_2}$ of area $M$ is the unique global minimizer (up to translations)
    of problem \eqref{eq:5.anOKprob} for $\gamma > 0$ sufficiently small.
    
\end{cor}

Clearly, an analogous statement holds for $L_M^{e_1}$ when $\D(K) = \cb{\pm e_1}$ with an appropriate notion of vertical flatness. 
Furthermore, by inspecting the proof of \cref{prop:5.stillmin},  one can likewise verify that the criterion applies when $\D(K) = \cb{\pm e_1, \pm e_2}$, 
so that both $L_M^{e_1}$ and $L_M^{e_2}$ are global minimizers (and there are no others). 
Establishing the global minimality of slanted lamellae, however, would require repeating the second variation analysis in \cite{Fio26}, 
which was carried out only in the flat horizontal case, and this goes beyond the scope of this paper.

\bigskip
\subsection*{Acknowledgments}
The authors are members of the GNAMPA group of INdAM.

\printbibliography

\end{document}